\documentclass[12pt]{article}
\usepackage{latexsym,amssymb,amsmath,amsfonts,amsthm,color}

\usepackage{mathtools}

\newtheorem{theorem}{Theorem}

\newtheorem{lemma}{Lemma}

\newtheorem{fact}{Fact}

\def\beq{\begin{equation}}
\def\eeq{\end{equation}}
\def\mn{\medskip\noindent}

\def\ep{\epsilon}

\def\square{\vcenter{\vbox{\hrule height .4pt
  \hbox{\vrule width .4pt height 5pt \kern 5pt
        \vrule width .4pt} \hrule height .4pt}}}

\def\var{\hbox{var}\,}

\def\fns{\footnotesize}

\definecolor{darkred}{rgb}{0.7,0,0}

\usepackage{graphicx}
\graphicspath{%
    {converted_graphics/}
    {/}
}
\begin{document}

\title{SIR epidemics on evolving graphs}
\author{Yufeng Jiang, Remy Kassem, Grayson York, \\
Matthew Junge, and Rick Durrett \\
\small Dept.\ of Math, Duke University, Box 90320, Durham NC 27708-0320}

\date{\today}						

\maketitle

\begin{abstract}
We consider evoSIR, a variant of the SIR model, on Erd\H os-Renyi random graphs in which susceptibles with an infected neighbor break that connection at rate $\rho$ and rewire to a randomly chosen individual. We compute the critical infection rate $\lambda_c$ and the probability of a large epidemic by showing that they are the same for the delSIR model in which $S-I$ connections are deleted instead of rewired. The final size of a large delSIR epidemic has a continuous transition. Simulations suggest that the final size of a large evoSIR epidemic is discontinuous at $\lambda_c$.
\end{abstract}    

\section{Introduction}

In the SIR model, individuals are in one of three states: $S=$ susceptible, $I=$ infected, $R=$ removed (cannot be infected). Often this epidemic takes place in a homogeneously mixing population. However, here, we have a graph that gives the social structure of the population; vertices represent individuals and edges a connection. $S-I$ edges become $I-I$ at rate $\lambda$, i.e., after a time $T$ with an exponential($\lambda$) distribution: $P(T > t ) = e^{-\lambda t}$. The  two versions of the model we consider differ in the length of time individuals remain infected. In the first, infections always last for time 1. In the second, infection times are exponential(1) distributed. Once individuals leave the infected state, they enter the removed state. Our main interest here is in the variant of the model in which $S-I$ edges are broken at rate $\rho$ and the susceptible individual connects to an individual chosen at random from the graph. We call this process evoSIR. To prove results for evoSIR it is useful to also study the variant delSIR in which edges are deleted at rate $\rho$.

Since the turn of the century, the complex networks community has studied systems in which the structure of 
a social network and the states of individuals coevolve. For a survey see \cite{GB}. Since the pioneering
2006 work of Holme and Newman \cite{HN}, much attention has focused on the evolving voter model. In that system, each individual
has one of two opinions, say 0 and 1. On each step of the simulation, one $0-1$ edge is picked at random
and given an orientation $(x,y)$. With probability $1-\alpha$, $x$ adopts the opinion of $y$, while with
probability $\alpha$, $x$ breaks the connection with $y$ and chooses an individual (i) at random from the graph (rewire-to-random)
or (ii) randomly from those with the same opinion (rewire-to-same). When $\alpha$ is close to 1, then the graph quickly breaks
into a large number of components of individuals with the same opinion. When $\alpha$ is small, 
a large component of like-minded individuals forms. An account of the properties of this model and references to 
earlier work can be found in Durrett et al.\ \cite{evov}. Recently Basu and Sly \cite{BaSly}  rigorously proved the existence
of a phase transition from rapid disconnection to prolonged persistence for dense graphs. However they have not been able to explain
the dramatic difference in the qualitative properties of the system under the two rewiring schemes. 

In 2006, Gross et al.\ \cite{GDB} introduced the evolving SIS model, which is similar to the model with exponential rates 
that we study here, except that $I$ individuals return to state $S$ at rate 1.  Using the pair approximation, they
were able to show that, if the rewiring rate is fixed and the infection rate is increased, then the model undergoes a
discontinuous phase transition in which a critical epidemic infects a positive fraction of the individuals. There are
a number of papers in the physics literature that study this model \cite{Zanette,ZanGus,JRS,SSP,SBM}, but there is no
proof of the result in \cite{GB}. 

When we began writing this paper, the only work on evolving evoSIR that we knew about were two papers of Volz and Meyers
\cite{VM1,VM2}. They considered an evolving SIR model on a random regular graph. However, they used a strange update rule, called ``neighbor exchange,''
in which two edges are chosen and their connections swapped. This preserves the constant degree of vertices, but does not
seem very realistic.

As we were finishing up the writing, we learned of three recent papers by Britton and collaborators that study the delSIR and evoSIR models with  exponential infection times. \cite{LBSB} studies the delSIR model on two versions of the configuration model. In the Molloy and Reed (MR) model \cite{MR} degrees $d_i$  are a deterministic sequence with specified asymptotic properties. In the Newman Strogatz Watts (NWS) models, degrees $d_i$ are i.i.d. In both situations to make the graph, $d_i$ half-edges are attached to $i$ and when $d_1+ \cdots + d_n$ is even the half-edges are paired at random. We will discuss results from this and the next two papers at appropriate times in the text.

\cite{BJS} studies a one parameter family of models (SIR-$\omega$) that interpolates between delSIR and evoSIR. In this model, $S-I$ connections are broken at rate $\omega$. The connections are rewired with probability $\alpha$, and remain broken with probability $1-\alpha$. They study the initial phase of the epidemic using branching process and pair approximation methods. Finally, \cite{LBSB} extends the results in \cite{BJS} and explores the implications of their results for epidemics. The main take home message is that measures taken by individuals to protect themselves (rewiring) can be detrimental to the population as a whole. That is, the final size of an SIR-$\omega$ epidemic can exceed the final size of the original SIR model. Their results are primarily based on simulation. They consider the NSW configuration model, a clique network, and two real networks: the collaboration network reconstructed from arXiv postings in general relativity, and in the Facebook ``social circles'' network.

Throughout this article our focus will be on networks modeled by Erd\H os-Renyi graphs with mean degree $\mu$. These are random graphs on $n$ vertices in which every pair of vertices is independently connected by an edge with probability $\mu /n$. We denote a graph sampled in this manner by $G(n,\mu/n)$. It is well known (see \cite{RGD}) that if $\mu > 1$, then $G(n,\mu/n)$ with high probability has a single giant component with $\Omega(n)$ vertices. 

When we say a sequence of events $(A_n)$ occurs with high probability, we mean that $P(A_n) \to 1$. This will often be abbreviated as \emph{whp}. A sequence of events $(A_n)$ occurs with positive probability if $\inf_n P(A_n)>0$. We often will abbreviate this as \emph{wpp}. Initially one randomly chosen vertex is infected and the rest susceptible. When we say that a large epidemic occurs, we mean that the number of removed individuals after there are no more infected individuals is $\Omega(n)$ wpp. The \emph{critical value} for an SIR process is the value $\lambda_c$ such that if $\lambda<\lambda_c$, then the probability of a large epidemic goes to 0 as the size of the graph $n \to\infty$. If $\lambda >\lambda_c$, a large epidemic occurs wpp.

\subsection{SIR with fixed infection times}

In this section we consider the usual SIR dynamics in which each infection lasts for exactly time 1. This case is simple because each edge will be $S-I$ (or $I-S$) only once. When that happens the infection will be transferred 
to the other end with probability 
\beq
\tau^f = P(T \le 1 ) = 1-e^{-\lambda}
\label{tauf}
\eeq
and the transfers for different edges are independent. Here the `$f$' in the superscript is for ``fixed time." Due to the last observation, we can delete edges with probability $e^{-\lambda}$ and the connected components of the resulting graph will give the epidemic sizes when one member of the cluster is infected.  We can have a large epidemic if and only if the reduced graph has a giant component. In the physics literature this idea is attributed to Grassberger (1983), in math to Barbour and Mollison (1990), and in complex networks to Newman (2002).

Throughout this paper we will make use of Poisson thinning: if the number of objects $N$ is Poisson with mean $\lambda$, and if we flip a coin with probability $p$ of heads to see if we keep each object, then the number of objects kept is Poisson with mean $\lambda\mu$. Using this, it is easy to show:

\begin{fact}
 If the original graph is Erd\H os-Renyi with mean degree $\mu$, then the reduced graph is Erd\H os-Renyi with mean degree $\mu\tau^f$. So, a large epidemic occurs with positive probability if $\mu\tau^f > 1$. If $z_0$ is the fixed point smaller than $1$ of the generating function 
\beq
G(z) = \exp(-\mu\tau^f (1-z)), 
\label{ftgf}
\eeq
then $1-z_0$ gives both the limiting probability an infected individual will start a large epidemic, and the fraction of individuals who will become infected when a large epidemic occurs.
\end{fact}

\noindent
Here and in what follows, things that we call facts are known results while those we call theorems are new.
The formula for the probability of a large epidemic comes from the fact that, in its early stages, the growth of the epidemic is well approximated by a branching process. See Section 2.2 in \cite{RGD}. The probability of a large epidemic is the probability the branching process does not die out, which is also $1-z_0$.

For the developments below, it is useful to sketch a more sophisticated approach due to Martin-L\"of \cite{ML}, who used this idea to prove a central limit theorem for the number of individuals infected when there is a large epidemic. Suppose that we have deleted the edges that the infection will not cross to produce  $G(n, \bar\mu/n)$ where $\bar\mu = \mu\tau^f$. The calculation is based on an algorithm that computes the component containing an arbitrary starting vertex which we label `1'. It begins with the removed set $R_0=\emptyset$, the active (or infected) set $A_0=\{1\}$, and the unexplored (or susceptible) set $U_0 =\{2, \ldots n \}$. Let $\eta_{j,i} = 1$ if there is an edge from $i$ to $j$ and 0 otherwise. At time $t$, if $A_t \neq\emptyset$ we pick an $i_t \in A_t$ and update the sets as follows: 
\begin{align*}
{\cal R}_{t+1} & =  {\cal R}_t \cup \{ i_t \} \cr
{\cal A}_{t+1} & =  {\cal A}_t - \{ i_t \} \cup \{ y \in {\cal U}_t : \eta_{i_t,y} = 1 \} \\
{\cal U}_{t+1} & =  {\cal U}_t - \{ y \in {\cal U}_t : \eta_{i_t,y} = 1 \}.
\end{align*}
When ${\cal A}_t=\emptyset$ we have found all of the members of the cluster containing 1 and the algorithm halts. If ${\cal A}_t=\emptyset$ and we have not yet found the giant component we select $i_t \in {\cal U}_t$ and continue.

Let $R_t = |{\cal R}_t|$, $A_t = |{\cal A}_t|$ and $U_t = |{\cal U}_t|$. Let ${\cal F}_t$ be the $\sigma$-field generated by the process up to time $t$. The number of connections from $i_t$ to $U_t$ is binomial($U_t,\bar\mu/n)$ so
$$
E( \Delta U_t | {\cal F}_t )  =  - U_t \frac{\bar\mu}{n} \quad \hbox{and}\quad
\var( \Delta U_t | {\cal F}_t )  =  U_t \frac{\bar\mu}{n} \left( 1- \frac{\bar\mu}{n} \right).
$$
Using the fact that $X_t = (1-\bar\mu/n)^{-t} U_t$ is a martingale and computing second moments one can easily prove (see Section 4.1):

\begin{fact}  \label{MLft}
As $n\to\infty$, $U_{[ns]}/n$ converges to $u_s = \exp(-\bar\mu s)$ uniformly on $[0,1]$.
\end{fact}

\noindent 
When $U_t + R_t=n$ we have $A_t=0$. This occurs at the $y_0>0$ that satisfies
$e^{-\mu\tau^f y_0} = 1-y_0.$ $y_0$ gives the fraction of sites in the giant component and $1-z_0$.

\subsection{Fixed time infections with rewiring}

We now introduce rewiring of $S-I$ edges at rate $\rho$, i.e., susceptibles break their connection with infected individuals and rewire to an individual chosen at random. In order for the infection to be transmitted along an edge, it must occur  before any rewiring and before time 1. To compute this probability, note that (i) the probability that infection occurs before rewiring is $\lambda/(\lambda+\rho)$ and (ii) the minimum of two independent exponentials with rates $\lambda$ and $\rho$ is an exponential with rate $\lambda+\rho$, so the transmission probability is
\beq
\tau^f_r = \frac{\lambda}{\lambda+\rho}(1-e^{-(\lambda+\rho)}).
\label{taufr}
\eeq
Here the `$r$' subscript is for ``rewire." Our first result shows that evoSIR has the same critical value as delSIR.

\begin{theorem} \label{ftcrit}
The critical value for the fixed time epidemic with rewiring is given by the solution of $\mu\tau^f_r=1$. 
Moreover, if $\lambda < \lambda_c$, then the ratio of the expected epidemic size in delSIR to the size in evoSIR converges to 1.
\end{theorem}

The formula for the critical value is easily seen to be correct for the delSIR since, by the reasoning above, there is a large epidemic if and only if the reduced graph in which edges are retained with probability $\tau^f_r$ has a giant component.  It is clear than the delSIR model has a larger critical value than evoSIR. Thus, we only have to prove the reverse inequality. Intuitively, the equality of the two critical values holds because a subcritical delSIR epidemic dies out quickly, so it is unlikely that a rewiring will influence the outcome. We prove this in Section \ref{sec:pfth2}. 

When $n$ is large, the degree distribution, which is Binomial($n-1,\mu/n$), is approximately Poisson with mean $\mu$. Due to Poisson thinning,  the number of new infections directly caused by one $I$ is asymptotically Poisson with mean $\mu\tau^f_r$, and hence has limiting generating function
\beq
\hat G(z) = \exp(-\mu\tau^f_r (1-z)).
\label{hatG}
\eeq

\begin{theorem} \label{ftperc}
If $z_0$ is the fixed point $< 1$ of $\hat G(z)$, then $1-z_0$ gives the probability of a large delSIR or evoSIR epidemic.
\end{theorem}

In the case of the delSIR model, $1-z_0$ is the fraction of individuals infected in a large epidemic. This proportion goes to 0
at the critical value $\mu_c = 1/\tau_r ^f =1$. In evoSIR, we conjecture, but are not able to prove that the limiting fraction infected is larger than the probability of a large epidemic. More surprisingly, as the simulations in Figures \ref{fig:constrho4} suggest, it is discontinuous at the critical value. 
Here, we have plotted the final size for a large number of
simulations at each parameter value, so there are many points near the $x$ axis that correspond to epidemics that died out. 

\begin{figure}[h] 
  \centering
  \includegraphics[width=2.98in,height=1.92in,keepaspectratio]{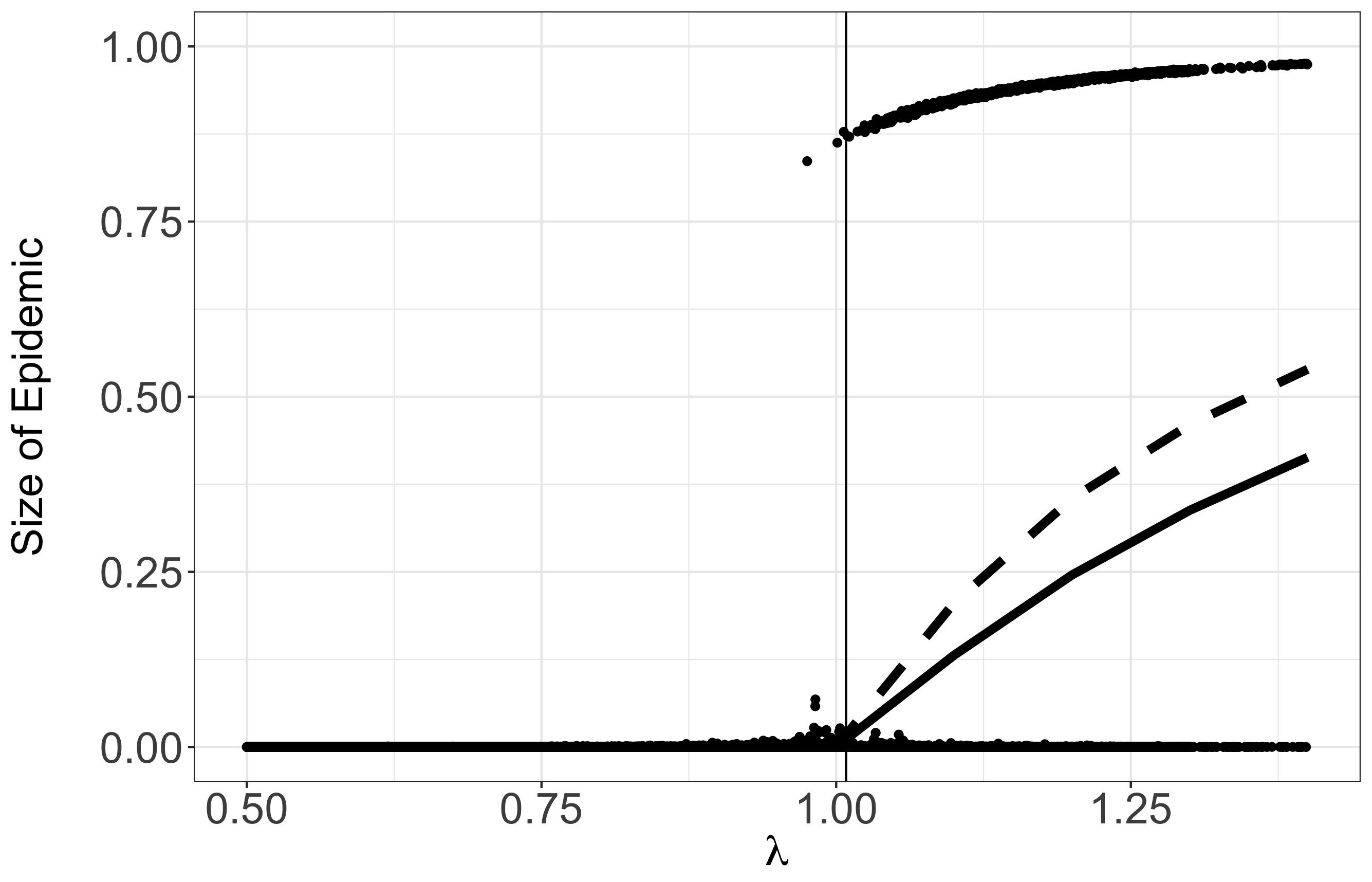}
  \includegraphics[width=2.98in,height=1.92in,keepaspectratio]{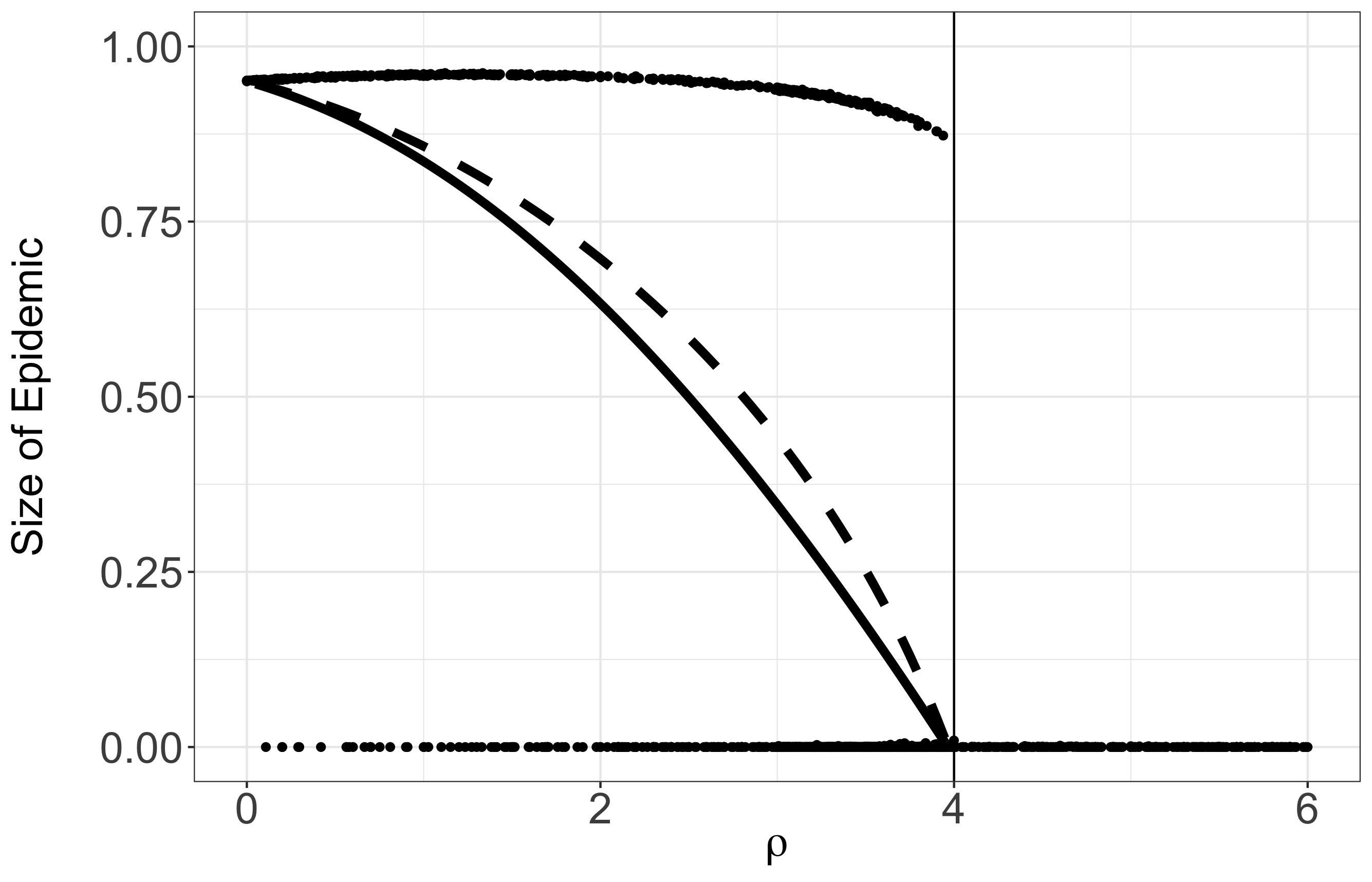}
 \caption{\fns Simulation of the fixed time evoSIR on an Erd\H os-Renyi graph with $\mu=5$. In the left panel, $\rho=4$ and 
$\lambda$ varies with  $\lambda_c \approx 1.0084$ in agreement with Theorem \ref{ftcrit}. 
The solid curve is the final size of the delSIR epidemic with the same parameters. The dashed line above it is an approximation that comes from Theorem \ref{finalsize}. In the right panel, $\lambda =1$ and $\rho$ varies with $\rho_c \approx 4$. Note that the final size is increasing for small $\rho$. This phenomenon is seen in  \cite[Figure 1]{LBSB}, but their simulation does not show a discontinuous phase transition.}
  \label{fig:constrho4}
\end{figure}

Figure \ref{fig:consttimecrit}, which gives a simulation of the epidemic at the critical value, 
gives some insight into why there is a discontinuity. Initially, the infection is critical but becomes supercritical
as degrees of susceptible nodes are increased by rewiring. As the graph shows, we get some large epidemics when $\lambda < \lambda_c$.  This is a finite size effect. As $n\to\infty$, it follows from Theorem \ref{ftperc} that the probability of a large epidemic in the subcritical case converges to 0.

\begin{figure}[h] 
  \centering
  \includegraphics[height=2.5in,keepaspectratio]{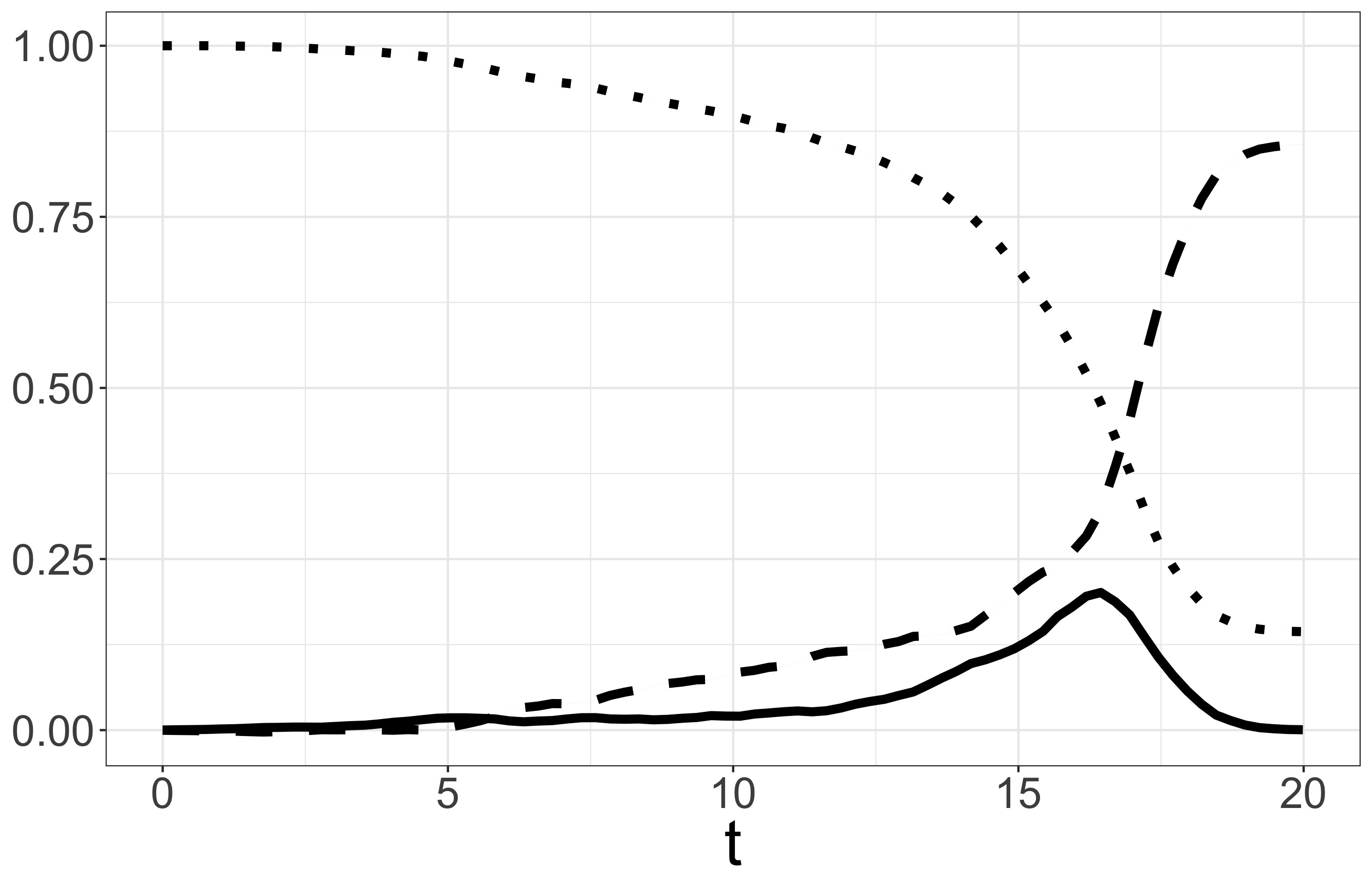}
  \caption{\fns Simulations of the fixed time evoSIR with $\mu=5$ and $\rho=4$ at the critical value $\lambda_c = 1.0084$. The solid curve
is the number of infecteds $I_t$. At the beginning it has slope 0, but due to rewiring the slope becomes positive, before the curve drops to 
zero due to the susceptible population being depleted. The dashed line is $R_t$, the dotted line $S_t$. }
        \label{fig:consttimecrit}
\end{figure}

\subsection{SIR epidemics with exponential infection times}

Now we suppose that the infection times are exponentially distributed with mean 1. The memoryless property of the exponential makes the process Markovian, but, as explained below, we lose the simple connection to percolation.  Suppose the infection time distribution $T$ has density $e^{-t}$. 
Introducing a superscript $e$ for exponential, the transmission probability is
\beq
\tau^e = 1 - \int_0^\infty dt \, e^{-t} e^{-\lambda t} = \frac{\lambda}{1+\lambda} 
\label{taue}
\eeq
i.e., the probability an exponential($\lambda$) infection time occurs before an exponential(1) recovery. 
The expected number of infections is $\mu\tau^e$ so the threshold 
for a large epidemic is
\beq
\lambda^e_c = \frac{1}{\mu-1}.
\label{expcrvcal}
\eeq

In the exponential case one cannot easily reduce the process to percolation since
the infection status of the edges going out of a vertex are correlated. Kuulasmaa (1982) came up with the
following construction. We replace each edge by a pair of oriented edges. For each vertex we have
dependent coin flips to determine whether we keep the edges. Kuulasmaa was able to prove some results
using this structure, but it is not easy to use. Here we mention it primarily to explain why the probability
of a large epidemic can be different from the fraction of individuals affected by one.

To compute the generating function of the number of infections directly caused by one infected, we note that if we condition on the value of the infection time $T$, we can conclude easily that the answer is
\beq
\hat G(z) = EG(e^{-\lambda T} + z[1-e^{-\lambda T}]).
\label{gengf}
\eeq
See \cite[Theorem 3.5.1]{RGD} for more details.
In the Poisson case $G(z) = \exp(-\mu(1-z))$ so writing this out gives 
\beq
\hat G(z) = e^{-\mu(1-z)} \int_0^\infty dt \, e^{-t} \exp(\mu(1-z)e^{-\lambda t}). 
\label{etgf}
\eeq

As in the fixed time setting, $\hat G$ can be used to compute the probability of a large epidemic.

\begin{fact}
If $z_0$ is the fixed point of $\bar G(z)$ that is $< 1$, then $1-z_0$ gives the probability of a large epidemic.
\end{fact}

In the case of exponential infection times, the final size of a large epidemic is not the same as the probability of a large epidemic.

\begin{theorem} \label{sizeexp}
Recall $\tau^e = \lambda/(1+\lambda)$.
The fraction of individuals infected in a large epidemic for SIR with unit exponential infection times is $1-z_0$ where $z_0$ is the fixed point $< 1$ of 
$$
\exp(-\mu\tau^e(1-z)) = z.
$$
\end{theorem} 

\noindent
This may be a known result, but we do not have a reference. We prove this in Section \ref{sec:MLexp} by using a variant of the calculation of Martin-L\"of \cite{ML}. The inequality $1+x<e^x$, implies $1/(1+\lambda) > e^{-\lambda}$ and it follows that 
$$
\tau^e  = 1- \frac{1}{1+\lambda} < 1 -e^{-\lambda} = \tau^f.
$$ 
Thus, the epidemic is larger for exponential infection times than for fixed times with the same mean.

\subsection{Limiting ODEs limit for SIR epidemics on graphs} \label{sec:ODE}

In this section we describe approaches to obtain ODE limits for SIR models on three random graphs: the complete graph, the configuration model, and the 
Erd\"os-Renyi random graph. As $n\to\infty$, the epidemic on the complete graph, i.e., the homogeneously mixing case, the system converges to an ODE:
\begin{align}
dS/dt & = -\beta SI/n 
\nonumber \\
dI/dt & = \beta SI/n - I 
\label{homomixODE}\\
dR/dt & = I
\nonumber
\end{align} 
See \cite[Section 7.3]{Allen}. 

The final size of the epidemic in the homogeneously mixing case is the same in the Erd\"os-Renyi graph,
but that does not mean that the ODE in \eqref{homomixODE} is correct for the epidemic on the Erd\H os-Renyi random graph. 
Volz (2008) was the first to derive a limiting ODE for an SIR epidemic on a graph generated by the configuration
model. Miller (2011) later simplified the derivation to produce a single ODE. 
Let $(u,v)$ be an oriented edge chosen at random from the graph and let $\theta(t)$
be the probability there has not been an infectious contact from $v$ to $u$. If we let $\psi(\theta)$ be the 
generating function of the degree distribution, then Miller's ODE is
\beq
\frac{d\theta}{dt} = - \beta \theta + \gamma(1-\theta) + \gamma \frac{\psi'(\theta)}{\psi(\theta)}
\label{MillerODE}
\eeq
where $\beta$ is the infection rate and $\gamma$ is the rate that infections become healthy. 
Given $\theta$, we have $S=\psi(\theta)$, $dR/dt = \gamma I$ and $I=1-R-S$.

The results of Volz and Miller were based on heuristic computations, but later their
conclusion was made rigorous by Decreusfond et al.\ (2012) and Janson et al.\ (2014). 
Here, we will derive an equation that is specific to the Erd\H os-Renyi graph, but that can more easily be extended to 
include rewiring.

When a site becomes infected, we add to the graph all edges connected to it that have not been revealed before. There will be no new edge connecting to a preexisting $I$ or an $R$, since the edge would have been revealed when that site first became infected. Edges to $S$ vertices will  exist with probability $\mu/n$ independent of what has happened before.

\begin{figure}[tbp] 
  \centering
  \includegraphics[width=3.07in,height=3.07in,keepaspectratio]{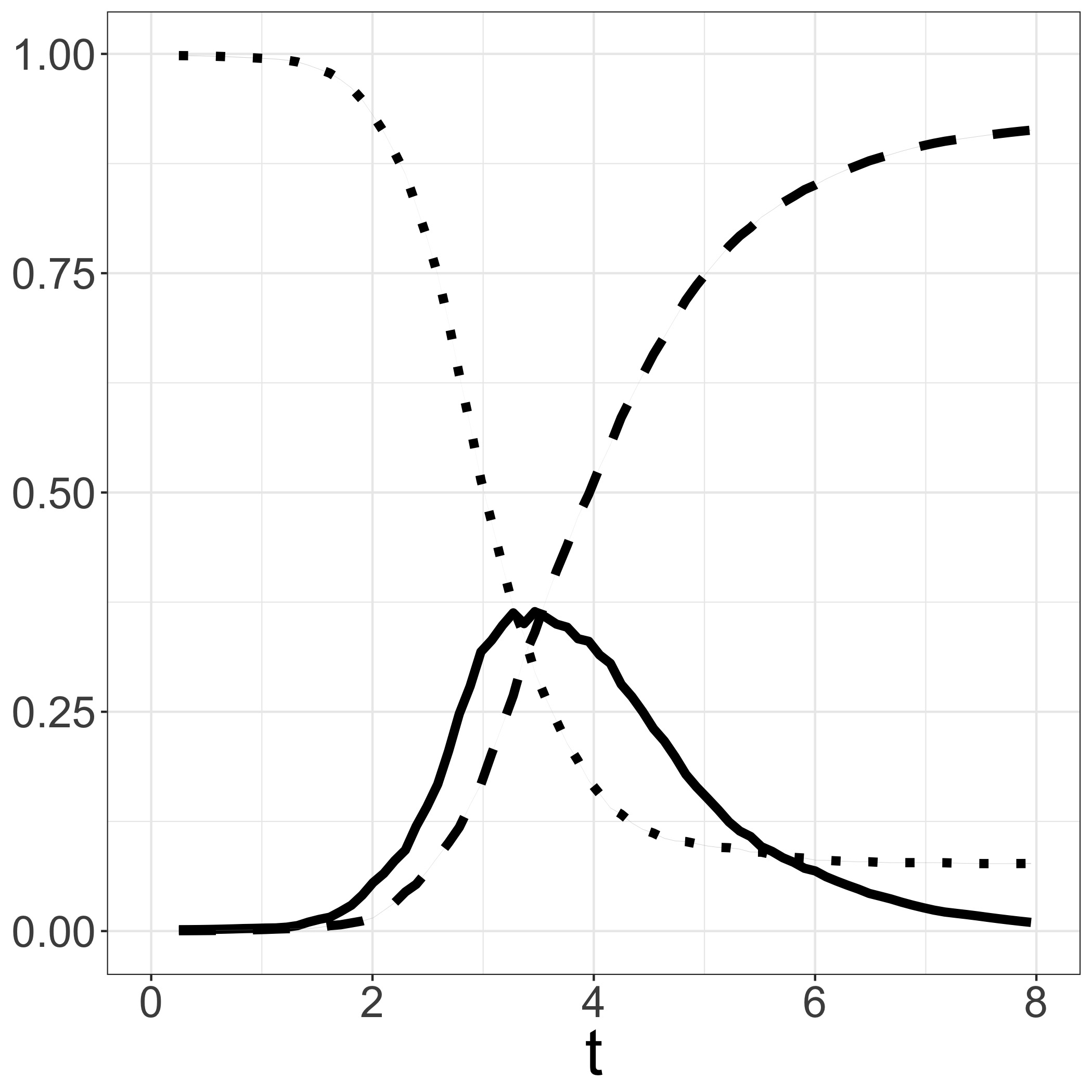}
  \includegraphics[width=3.07in,height=3.07in,keepaspectratio]{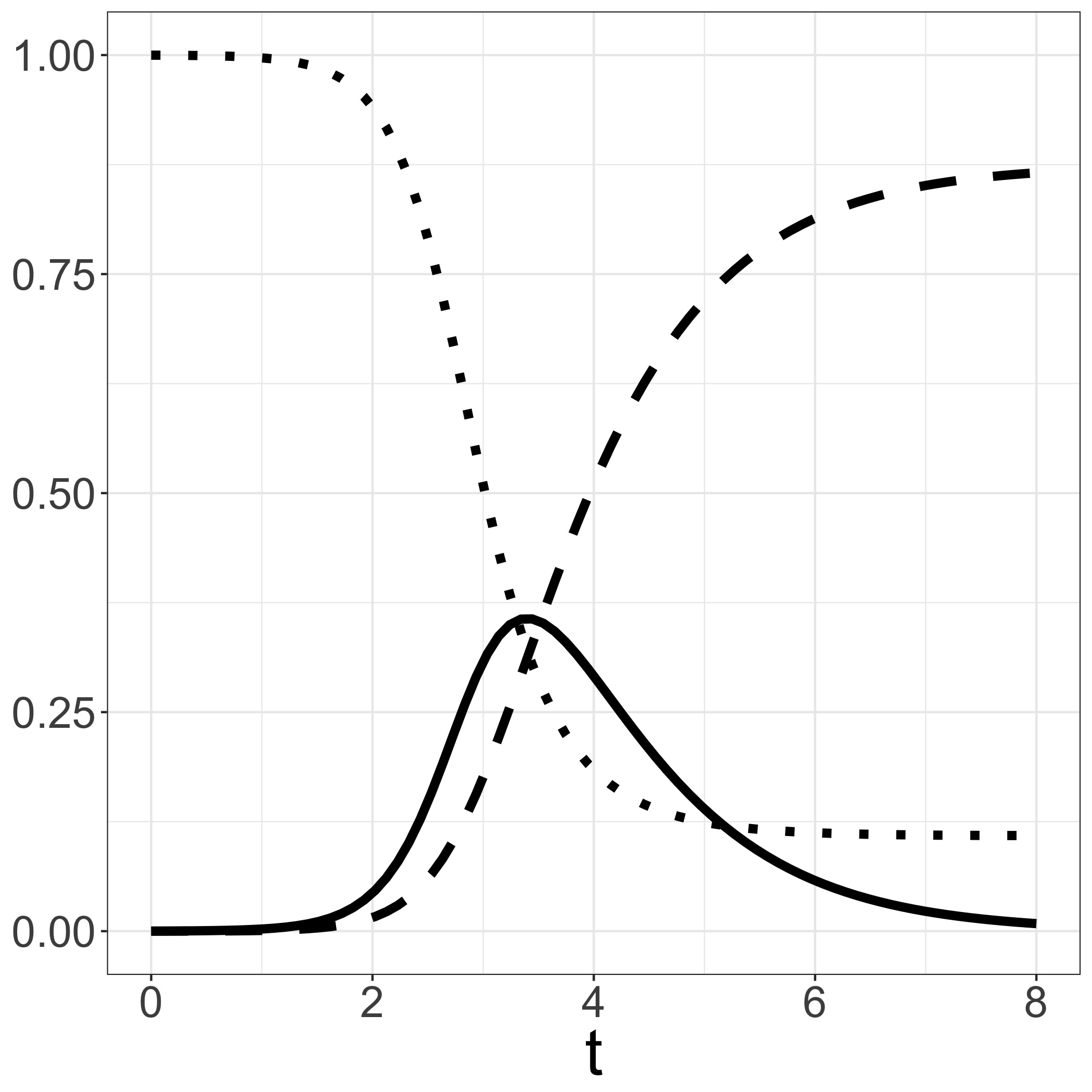} 
  \caption{\fns Comparison of simulation (left panel) with solution of ODE (right panel). The dotted line is $S_t$.
The solid curve is $I_t$. The dashed line is $R_t$.}
  \label{fig:ode3}
\end{figure}

Let $S_k$ be the number of susceptibles connected to $k$ infecteds. The unexplored vertices $S_0$ are not yet in the graph. 
If we let $S_{-1}=0$ and $F=\sum_k k S_k$, then the main equation can be written as:
\beq
\frac{dS_k}{dt} = - \lambda k S_k + \lambda F \frac{\mu}{n} (S_{k-1} - S_k) +  [(k+1) S_{k+1} - k S_k].
\label{SkODE}
\eeq
In words, an $S_k$ turns into an $I$ at rate $\lambda k$. The total rate at which new infections happen is $\lambda F$. When a new infection is created, it will be connected to an existing susceptible with probability $\mu/n$. This promotes an $S_j$ to $S_{j+1}$. In addition, infected edges become removed at rate 1. Thus $S_j$'s are demoted at rate $j$ to become $S_{j-1}$. The other two equations are simple
$$
\frac{dI}{dt}  = \lambda F - I \qquad
\frac{dR}{dt}  = I.
$$
Figure \ref{fig:ode3} shows that the simulation and differential equation agree. We have also verified that these curves agree with the
the solution of the Miller-Volz ODE.

Ball et al.~\cite{BBLS} do an in depth analysis of the delSIR model on MR and NSW configuration models. In \cite[Section 3]{BBLS}, they obtain an ODE limit that is related to
\eqref{SkODE}, but is more detailed because it considers the degrees of $S$, $I$, and $R$ vertices. This approach leads to results about the final size of a delSIR epidemic that include a central limit theorem.

\subsection{Exponential infection time with rewiring} 

The minimum of the recovery and rewiring times $T_r$ is exponential($1+\rho$) so
$$
\tau^e_r  = 1 - \int_0^\infty dt \, (1+\rho) e^{-(1+\rho)t} e^{-\lambda t} = \lambda/(\lambda+1+\rho)
$$
where again the subscript $r$ stands for ``rewire.'' Based on this reasoning,
the critical value is the solution to $\mu\lambda/(\lambda+1+\rho) = 1$. If $\mu$ and $\rho$ are fixed, solving gives
\beq
\lambda_c = \frac{1+\rho}{\mu-1}.
\label{cvSIRr}
\eeq

A result for configuration model graphs is given in (1) of Britton et al.~\cite{BJS}. They show that if the mean degree is $\mu$ and its variance is $\sigma^2$, then the branching approximation to the basic reproduction number is
\beq
R_0^{BA} = \frac{\beta}{\beta + \gamma + \omega} \left( \mu - 1 + \frac{\sigma^2}{\mu} \right).
\label{R0BA}
\eeq
If we change their notation to ours, the first term becomes $\lambda/(\lambda+1+\rho)$.
The expression in parentheses is the mean of the size biased degree distribution $q_{j-1} = kp_k/\mu$. When the degree distribution is
Poisson($\mu$), the size-biased degree distribution is also Poisson($\mu$). This implies that the quantity in parentheses is $\mu$,
but one can check that more directly by noting that when the degree distribution is Poisson $\sigma^2=\mu$.

The proof of Theorem \ref{ftcrit} generalizes easily to show

\begin{theorem} \label{expcrit}
\eqref{cvSIRr} gives the critical value for the delSIR and evoSIR models with exponential infection times. 
\end{theorem}

This result is related to a remark in \cite{BJS}, ``We note that \eqref{R0BA} is independent of $\alpha$ (the fraction of edges that are
rewired), so it has no effect on the beginning of an outbreak if rewired edges are dropped, always attached to a new susceptible
or a mixture of the two.'' Here, we go beyond that heuristic and prove that the two critical values are equal.

To compute the generating function of the number of infections, we note that if we condition on the value of $T$,
the probability of transmission is
$$
\tau(T)  = \int_0^T dt \, \lambda  e^{-\lambda t} e^{-\rho t}  
= \frac{\lambda}{\lambda+\rho} (1-e^{-(\lambda+\rho)T})
$$
i.e., (a) infection occurs before rewiring and (b) the minimum of the infection and rewiring times occurs before time $T$. 
Using \eqref{gengf} and letting  $\mu_r = \mu\lambda/(\lambda+\rho)$, we see that in the Poisson case 
\beq
\bar G(z) = e^{-\mu_r(1-z)} \int_0^\infty dt \, e^{-t} \exp(\mu_r(1-z)e^{-(\lambda+\rho)t}).
\label{efpeq2}
\eeq

\begin{theorem} \label{exprwcrit}
If $z_0$ is the fixed point of $\bar G$, then the probability of a large epidemic is $1-z_0$
in delSIR  and evoSIR with mean one exponential infection times.
\end{theorem}

Again, the proof follows easily from the proof of Theorem \ref{ftperc}.
In the delSIR model, the fraction of the population in a large epidemic is the same as the probability
of a large epidemic and hence has a continuous transition. These two quantities are different in the evoSIR model. 

As the simulations in Figure \ref{fig:exptimerho4} show,
the final size of the evoSIR epidemic is discontinuous at the critical value.
 Here we have plotted the final size for a large number of simulations at each critical value, so there are many points near the $x$ axis that correspond to epidemics that died out.

\begin{figure}[tbp] 
  \centering
  \includegraphics[width=2.98in,height=1.92in,keepaspectratio]{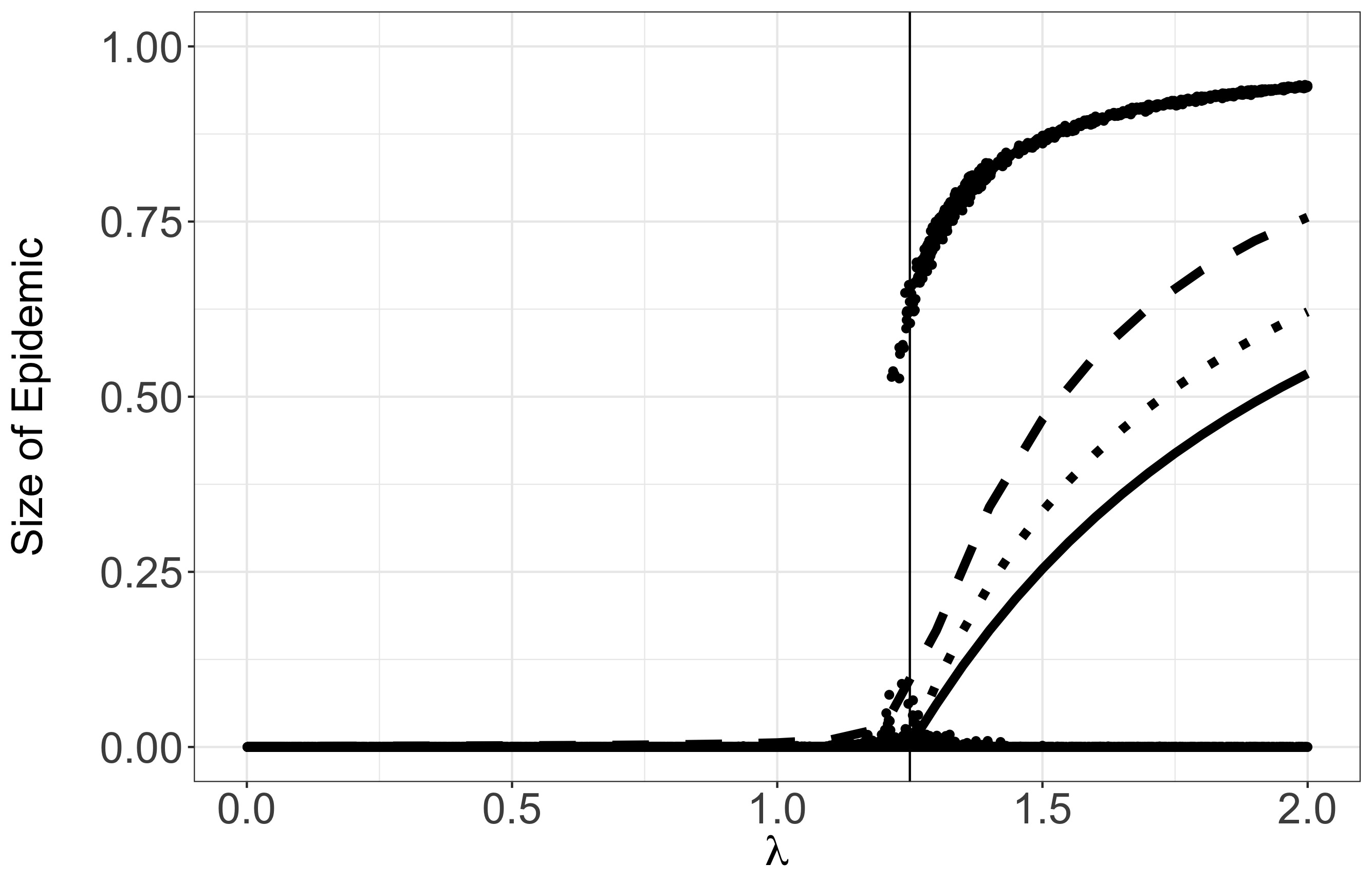}
  \includegraphics[width=2.98in,height=1.92in,keepaspectratio]{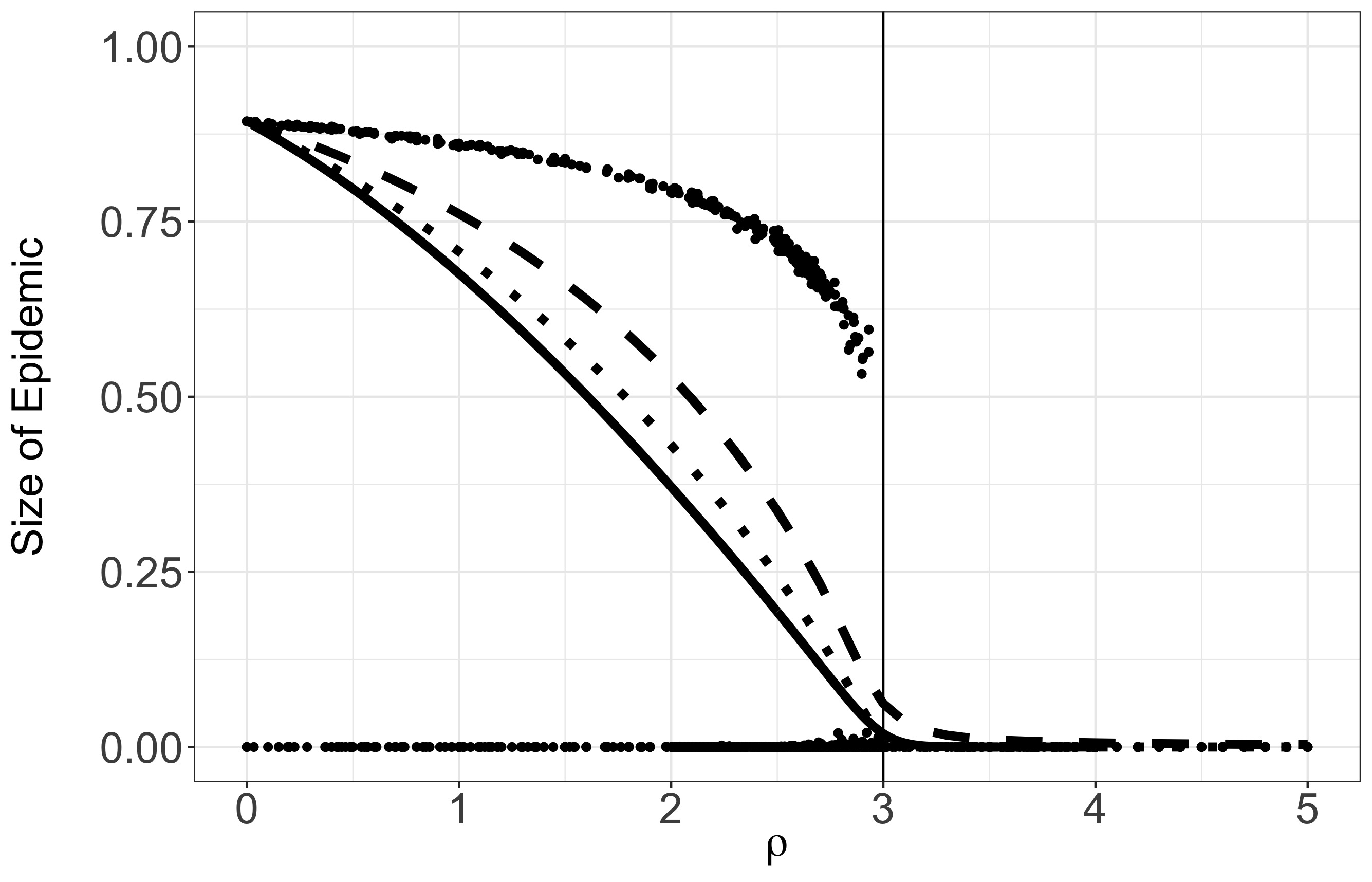}
 \caption{\fns Fraction of individuals infected in an SIR epidemic on an Erd\"os-Renyi graph with $\mu=5$. In the left panel, $\rho=4$ and $\lambda$ varies. In agreement with \eqref{cvSIRr}, $\lambda_c=1.25$. The solid curve is the final size of the delSIR epidemic with the same parameters. The dotted line above it is an approximation that comes from Theorem \ref{finalsize}. The third curve comes from solving (4). In the right panel, $\lambda=1$ with $\rho$ varying. $\rho_c = 3$ in agreement with \eqref{cvSIRr}. The assumptions are similar to Figure 1 of \cite{LBSB}, but our curve is decreasing and their simulation does not show a discontinuous phase transition. If $\lambda$ is larger we do see an increase in density for small $\rho$.}
  \label{fig:exptimerho4}
\end{figure}

\begin{figure}[h] 
  \centering
  \includegraphics[height=2.5in,keepaspectratio]{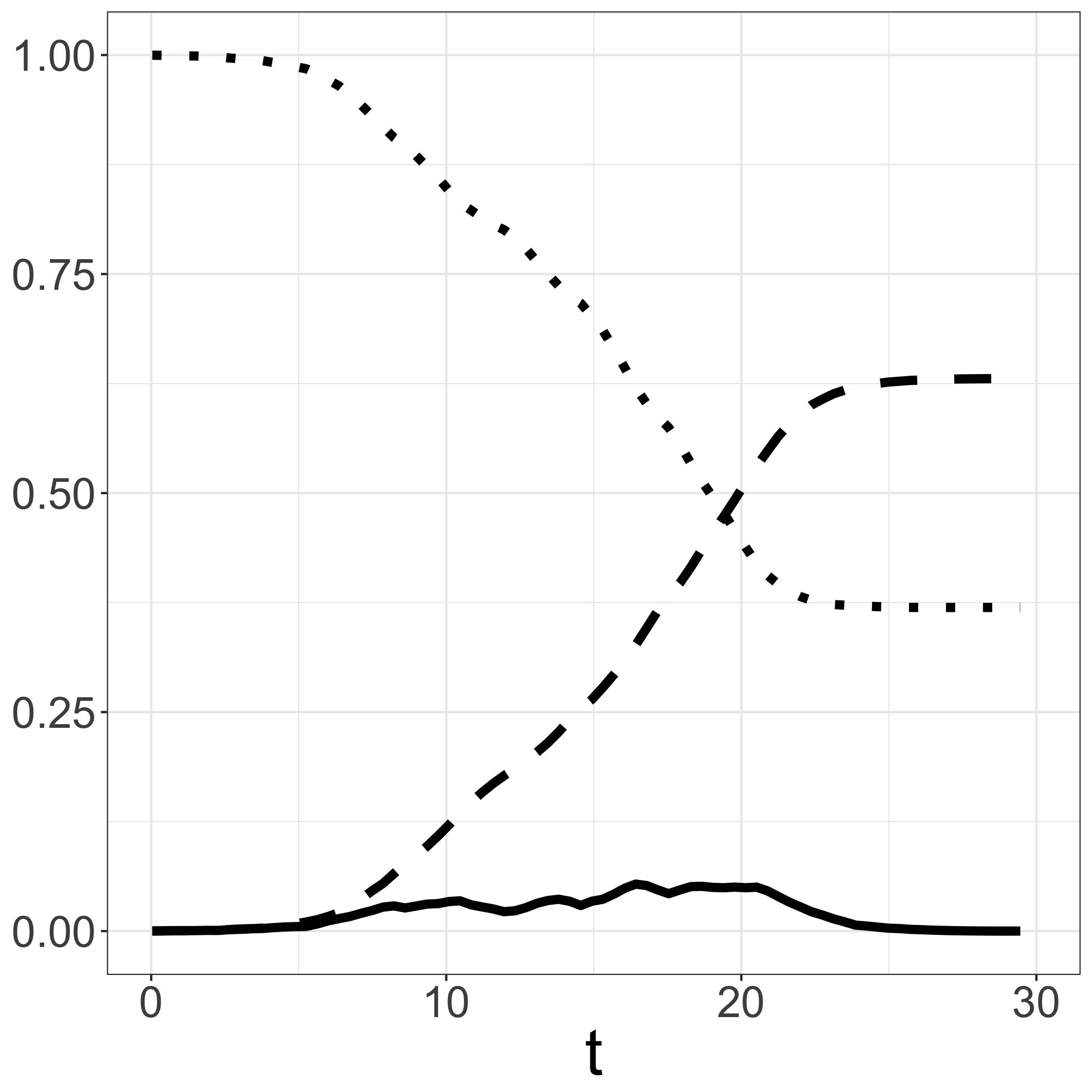}
  \caption{\fns Simulations for the fixed time epidemic with $\mu=5$ and $\rho=4$ at the critical value $\lambda_c = 1.25$. The solid curve
is the number of infecteds $I_t$. At the beginning it has slope 0, but due to rewiring the slope becomes positive, before the curve drops to 
zero due to the susceptible population being depleted. The dashed line is $R_t$, the dotted line $S_t$. }
  \label{fig:critexp}
\end{figure}

\subsection{Attempts at a rigorous analysis} \label{sec:attempt}

Here we explain our efforts to model evoSIR using (i) our extension of Martin-L\"of's ideas and (ii) a modification of the ODE in \eqref{SkODE}. We provide these details in the hope someone can find an accurate approximation that explains the reason for the discontinuous phase transition. 

\subsubsection{Modifying Martin-L\"of's calculation}

The key to Martin-L\"of's proof described in Section 1.1 is to show that $U_{[ns]}/n \to u_s$ which solves
$$
\frac{du_s}{ds} = - \mu\tau u_s.
$$
To take into account the rewiring in the fixed time case, we will let $v_s$ be the average degree of unexplored vertices at time $s$.
Repeating the reasoning in Section \ref{sec:MLfixed} we arrive at differential equations 
\begin{align}
\frac{du_s}{ds}& = - v_s \tau u_s (1-\alpha^f) \label{MLrw}\\
\frac{dv_s}{ds}& = v_s\tau u_s \alpha^f \nonumber
\end{align}
where $\alpha^f = 1 -\tau^f_r/\tau^f$ is the probability a rewiring prevents an infection. See Section 4.3 for more details. The intuition behind the second equation is that edges are being rewired at rate  $v_s \tau u_s \alpha^f$ and are attached to randomly chosen vertices, so the average degree increases at this rate. Note that here time is the number of vertices that have been exposed, which is not the right time scale so there is no guarantee that this computes the right answer.

Dropping the superscript and combining the two equations gives
$$
\frac{d}{ds}[ \alpha u_s + (1-\alpha) v_s ] = 0.
$$
So $\alpha u_s + (1-\alpha) v_s = \alpha + (1-\alpha) \mu$, and we can reduce the system to one equation. Solving them in Section \ref{sec:rewire} gives
\beq
u = \frac{A}{B + (A-B) e^{At}},
\label{solni}
\eeq
where $A = \tau[u(1-\alpha) + \alpha]$ and $B = \tau\alpha$. 

As in Martin-L\"of's result.

\begin{theorem} \label{finalsize}
Our approximation to the final size of a large epidemic with fixed infection times and rewiring is given by the solution $>0$ of $u(t) = 1-t$.
\end{theorem}

To generalize to exponential infection times, 
we replace $\tau^f = 1-e^{-\lambda}$ by 
$$
\tau^e = E\tau = E(1-e^{-\lambda T}) = \frac{\lambda}{\lambda+1},
$$
and replace $\alpha^f$ by 
$$
\alpha^e = 1 - \tau_r^e/\tau^e = \frac{\rho}{\rho+1+\lambda}.
$$

\begin{figure}[tbp] 
  \centering
  \includegraphics[bb=53 57 739 555,width=4in,height=2.9in,keepaspectratio]{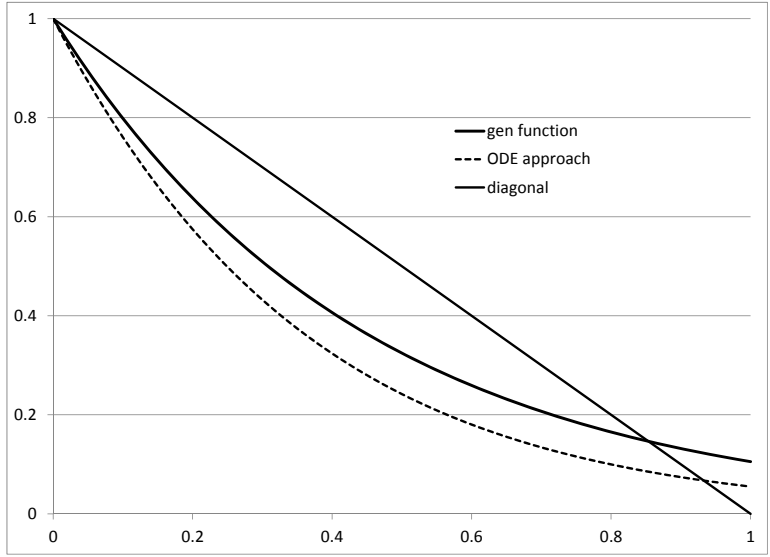}
\caption{\fns Calculation of (1) and (8) for the case $\mu\tau^f_r=3$, $\alpha=1/4$.  } 
  \label{fig:gf_ODE}
\end{figure}

\noindent
Figure \ref{fig:gf_ODE} compares the solution of the new differential equation \eqref{solni} with $\exp(-\mu\tau^f s)$. The place were the solution 
crosses the diagonal gives the approximation in Theorem \ref{finalsize}. This is the quantity plotted as dashed lines in Figures 1 and 4. Note that
in either case the approximation is continuous at $\lambda_c$. To see this is true in general, note that by \eqref{MLrw}
\begin{align*}
\frac{d}{ds} u_s v_s &= - v_s \tau u_s(1-\alpha) v_s + u_s v_s \tau u_s \alpha \\ 
& = u_s v_s [-\tau(1-\alpha) v_s + u_s \alpha].
\end{align*}
When $s=0$, $v_s=\mu$. In the critical case $\mu\tau(1-\alpha)=1$, so the quantity in square brackets is $<0$.
As $s$ increases, $u_s$ decreases and $v_s$ increases, so $u_sv_s$ is decreasing.

\subsubsection{Modifying the ODE}

By analogy with what we did to Martin-L\"of's equation, we should let $\mu_t$ be the average degree of susceptibles at time $t$
and modify the ODE from \eqref{SkODE} to become
\begin{align*}
\frac{dS_k}{dt} &=  - \lambda k S_k + \lambda F \frac{\mu_t}{n} (S_{k-1} - S_k) +  [(k+1) S_{k+1} - k S_k] \\
&\qquad \qquad  + \rho\left(I-\frac{I}{n}\right)  [(k+1) S_{k+1} - k S_k] \\
\frac{d\mu_t}{dt} & = \frac{\rho F}{n} \left(1-\frac{I+R}{n}\right).
\end{align*}
Figure \ref{fig:exprwsim} compares this ODE with the results of simulation. The shapes of the curves are similar but
the magnitudes do not agree. This situation is somewhat puzzling since the ODE seems to accurately model the dynamics.

\begin{figure}[tbp] 
  \centering
  \includegraphics[width=3.07in,height=3.07in,keepaspectratio]{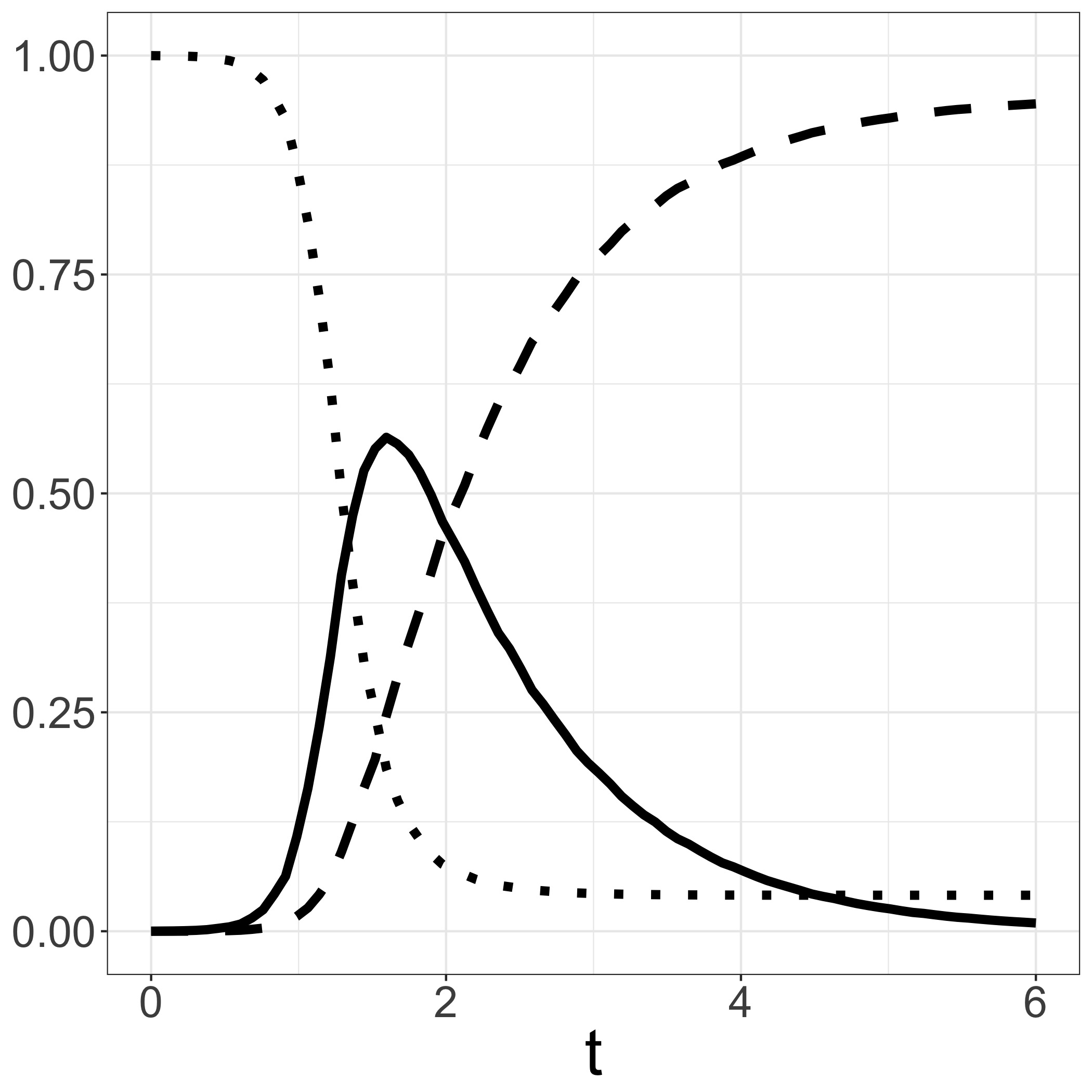}
  \includegraphics[width=3.07in,height=3.07in,keepaspectratio]{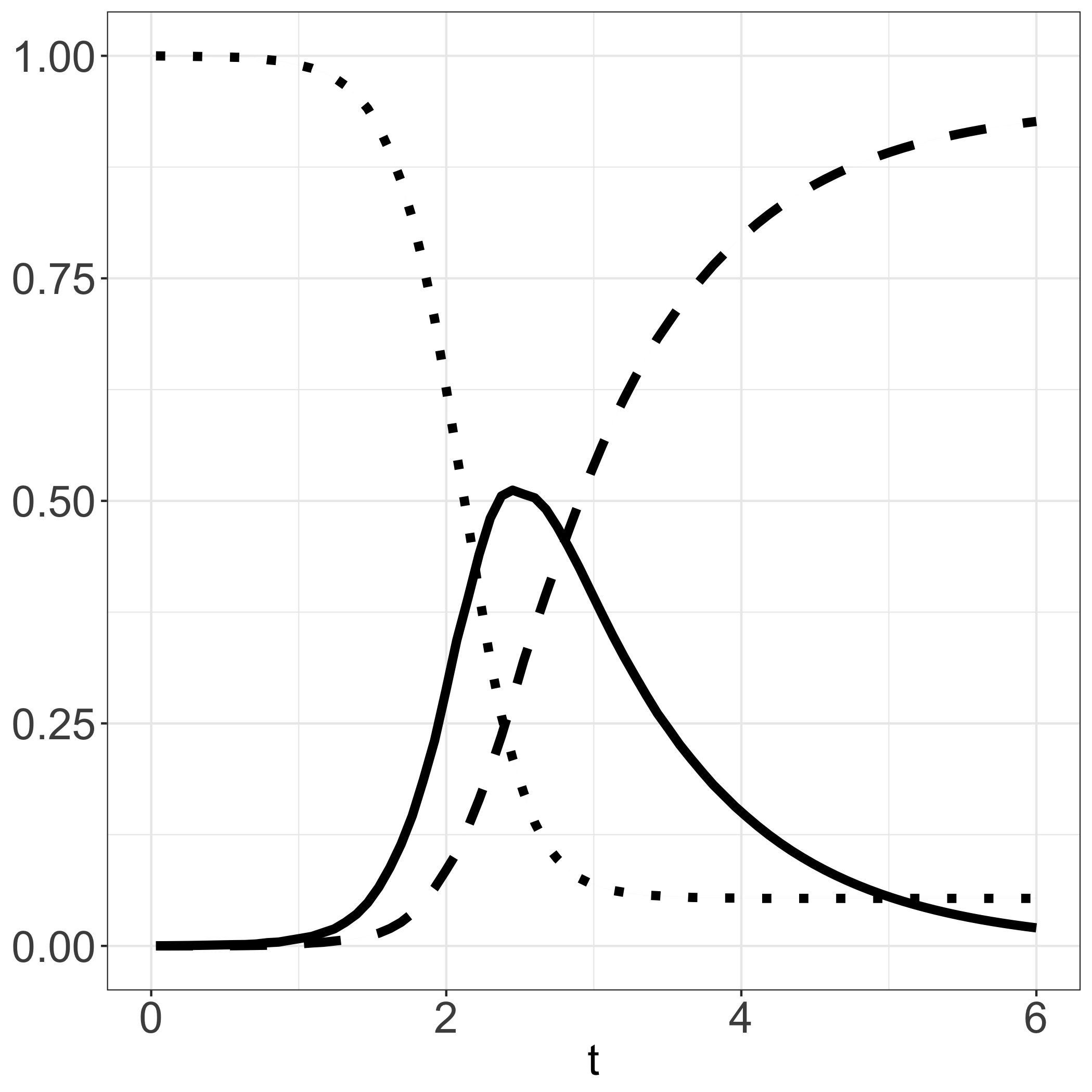}
  \caption{\fns Comparison of the ODE with rewiring (left panel) with the results of simulation (right panel). The curves have similar shapes but the time scales do not agree and the maximum number of infecteds is different. }
  \label{fig:exprwsim}
\end{figure}

\section{Proof of Theorems \ref{ftcrit} and \ref{expcrit}} \label{sec:pfth2}

\mn
To avoid interrupting the flow of the proof, we begin by proving some simple facts that will be useful in the proof. The proof of Theorem \ref{ftcrit} and the extension of the argument that proves Theorem 4 is given at the end of the section.

\begin{lemma} \label{xtok}
If $0 < x < 1$ and $k$ is a positive integer, then $(1-x)^k \ge 1-kx$.
\end{lemma}

\begin{proof}
Let $A_1, \ldots A_k$ be independent and have probability $1-x$.
$$
(1-x)^k = P( \cap_{i=1}^k A_i ) \ge 1 - \sum_{i=1}^k P(A^c_i) = 1-xk
$$
proving the desired result.
\end{proof}

\begin{lemma} \label{remy}
If $Z=\hbox{binomial}(n,p)$ with $0<p<1$ then $P(Z \ge 2) \le P(Z \ge 1)^2$.
\end{lemma}

\begin{proof}
Let $X_1, \ldots X_n$ be independent with $P(X_i=1)=p$ and $P(X_i=0)=1-p$. Let
\begin{align*}
N_1 & = \inf \{ m \in [1,n] : X_m = 1 \} \\
N_2 & = \inf \{ m \in (N_1,n] : X_m = 1 \}
\end{align*}
where $\inf \emptyset = \infty$. Notice that
\begin{align*}
P(N_1 < \infty) & = P(Z \ge 1) = 1-(1-p)^n \\
P(N_2<\infty \mid N_1 = m) & = 1 - (1-p)^{n-m} \le P(Z \ge 1).
\end{align*}
Since the last result holds for all possible values of $N_1$, we have 
$P(Z\ge 2 \mid Z \ge 1) \le P(Z\le 1)$ and the desired result follows.
\end{proof} 

\begin{lemma} \label{ldev}
Let $\mu>0$ and $Z=\hbox{binomial}(n,\mu/n)$. It holds that 
$$
P \left( Z > \frac{3}{\ln 2} \log n \right) \le \frac{\exp(\mu)}{n^3}.
$$
\end{lemma}

\begin{proof}
The moment generating function for a binomial is
$$
Ee^{\theta Z} = \sum_{m=0}^n \binom{n}{m} p^m(1-p)^{n-m} e^{\theta m} = (1-p+pe^{\theta})^n.
$$
Using Markov's inequality, then taking $p=\mu/n$ we have
$$
e^{\theta C \log n} P( Z > C\log n ) \le (1 + p(e^{\theta}-1))^n \le \exp(\mu(e^{\theta}-1))
$$
since $1+x \le e^x$.
Taking $\theta = \log 2$ and $C = 3/(\log 2)$ gives the desired result.
\end{proof}

Consider evoSIR on $G= G(n,\mu/n)$ with rewire rate $\rho$, infection rate $\lambda$, and in which infection lasts for a fixed time 1. The probability an infection is transmitted from an infected vertex to a susceptible neighbor is 
$$
\tau_r  = \frac{\lambda}{\lambda + \rho} (1-e^{-(\lambda+\rho)}).
$$
To analyze this process we will consider the deletion model (delSIR) in which edges that are rewired in evoSIR are instead deleted. 
As mentioned earlier, a simple coupling gives the final set of removed individuals in delSIR is contained in the analogous set for evoSIR with the same parameters. In our notation,
$$
\lambda_c(\hbox{evoSIR}) \le \lambda_c(\hbox{delSIR}).
$$ 

To compare the two dynamics we will let 
$$
\tau = 1 - e^{-\lambda}
$$ 
be the probability an infection will be transmitted to a neighbor. We start with $G(n,\mu\tau/n)$. Let 
$$
\alpha = 1 - \tau_r/\tau
$$ 
be the probability that rewiring eliminates a successful infection event. 
In the delSIR dynamic, edges are deleted from $G(n,\mu\tau/n)$ with probability $\alpha$ while in evoSIR they are rewired with probability $\alpha$.
The compare the two evolutions we will first run the delSIR epidemic to completion. Once this is done we will randomly rewire the edges deleted in delSIR. If the rewiring creates a new infection, then we have to continue to run the process and make some further estimates. 

Let ${\cal R}'$ be the set of sites that are removed at time $\infty$ in delSIR, and let ${\cal R}$ be the set of removed sites at time $\infty$ in evoSIR. Let $R' = |{\cal R}'|$ and $R = |{\cal R}|$. 

\begin{lemma} \label{rtail}
If $\tau \mu\alpha  < 1$, then there are constants $C_1$ and $C_2$ so that
\beq
P( R > C_1 \log n ) \le C_2 n^{-3/2}
\label{Rbd}
\eeq
and $E R = E R' +o(1).$ 	
\end{lemma}

\begin{proof}
Let $\gamma = \mu\tau - 1 - \log(\mu\tau) > 0$. The proof of \cite[Theorem 2.3.1]{RGD} implies that
\beq
P( R' \ge (3/\gamma) \log n ) \le n^{-3}/\mu\tau\alpha
\label{ldR'}
\eeq
(To see this look at the last displayed equation in the proof and take $\ep=2$, $\lambda=\mu\tau$.)

Let $D = \max_{v \in G} \deg v$. Lemma \ref{ldev} implies
\beq
P \left( D > \frac{3}{\ln 2} \log n \right) \le \frac{\exp(\mu\tau\alpha)}{n^2}
\label{ubD}
\eeq
Let $X'$ be the number of deleted edges in evoSIR. Clearly $X' \le D R'$.
After delSIR has been run to fixation use independent random variables
independent of delSIR to randomly rewire the deleted edges.
Let $Y$ be the number of edges that rewire to ${\cal R}'$.  By construction
$$
Y = \text{binomial}(X',R'/n).
$$
Combining \eqref{ubD} and \eqref{ldR'} we see that there are constant $C_3$ and $C_4$ so that
if $G = \{D ,R' \leq C_3 \log n\}$ then 
$$
P(G) \geq 1 - C_4/n^2
$$ 
On $G$ we have
$$
Y \preceq \text{binomial}(DR',R'/n) \preceq  \text{binomial}(C_3^2 \log^2 n,C_3 \log n/n) \equiv \bar Y
$$
where $\equiv$ indicates that the last equality defines $\bar Y$ and
$Y\preceq Z$ denotes stochastic order: 
$$
P(Y> x) \le P(Z>x)\quad\hbox{for all $x$}.
$$
Using the formula for the binomial distribution  
\beq
P(\bar Y=0) = \left(1- \frac{C_3 \log n}n \right)^{C_3^2 \log^2 n} \ge 1 - \frac{C_3^3 \log^3 n}{n}
\label{Y0bd}
\eeq
by Lemma \ref{xtok}. From this and Lemma \ref{remy} we get 
$$
P(\bar Y\ge 1 )  \leq \frac{C_3^3 \log^3 n}{n} \qquad
P( \bar Y \ge 2 )  \leq \frac{C_3^6 \log^6 n}{n^2}.
$$

Since the second inequality is $o(n^{-3/2})$, we can ignore the possibility that $Y\ge 2$. If $Y=0$ the evoSIR cluster coincides with the delSIR cluster and we are done. It remains to consider the case in which $Y=1$. One endpoint of the rewired edge is in ${\cal R}'$. Let $z_0$ be the one that is not. Note that by construction $z_0$ is chosen at random from $\{1, 2, \ldots n\} - {\cal R}'$. In the second stage of the process we will again run the delSIR dynamics, let ${\cal R}'' \subset \{1, 2, \ldots n\} - {\cal R}'$ be the new members of the removed set, and then flip coins to rewire the edges that have been deleted. Let $Y^*$ be the number of rewired edges that connect to ${\cal R}' \cup {\cal R}''$. The probability that $Y \ge 1$ and $Y^* \ge 1$ is $\le n^{-3/2}$ so only have to look at the case $Y^*=0$, i.e., we only need to look at the delSIR cluster. At this point we are almost ready to claim that the analysis of the second stage follows from that of the first. The last detail is to note that edges rewired in the first stage have increased the density of the graph, so the second stage takes place on an Erd\H os-Renyi graph on $\{1,2,\ldots n\}-{\cal R}'$ in which edges are open with probability $\rho/n$ where $\rho > \mu\tau$. The total number of rewired edges is $\le C \log^2 n$ with high probability, so if $n$ is large $\rho\alpha < 1$, and the reasoning from stage one applies.

To prove that $ER = ER'+o(1)$ we write $E(X;A)$ for the integral of $X$ over $A$ and note that

\begin{itemize}
  \item On $Y=0$ we have $R=R'$.
  \item $P(Y\ge 2) \le (C_3^6 \log^6 n)/n^2$ so we can use the trivial bound $R \le n$ to conclude $E(R;Y \ge 2) \to 0$.
  \item $P(Y=1, Y^*\ge 1)\le C_3^6 \log^2 n)/n^2$, so we can again use the trivial bound $R \le n$.
  \item Finally on $\{ Y=1, Y^*=0 \}$ we let $R_1$ and $R_2$ be the contributions of the two stages. Since $R_2$ is independent of $\{Y_1=1\}$
$$
E((R_1+R_2; Y=1, Y^* = 1) \le E(R'; Y=1) + P(Y=1) ER'
$$
The second term $\to 0$ as $n \to\infty$. To handle the first term note that (2.3.4) in \cite{RGD} shows that $P(R'> k)$ tends to 0 exponentially fast.
This implies that $E(R')^2 \le C$ independent of $n$. Since $P(Y=1)$ an elementary argument shows $E(R';Y=1) \to 0$. (Divide the event into
two pieces depending on whether $R'\le m$ or $R > m$ or quote Theorem 1.6.8 in \cite{PTE4}.) 
\end{itemize}   

Combining the four conclusions proves that $ER = ER' + o(1)$. To prove \eqref{Rbd} note that the events in cases two and three have 
combined probability $\le n^{-3/2}$, and use \eqref{ldR'} for cases one and four. In the fourth case there is no need to condition
on $Y_1=0$. We simply use the fact that in this case we $Y= 1$ so $R \le R'_1 + R'_2$ where $R'_2$ is independent of $R_1'$
and then use \eqref{ldR'} again.
\end{proof}


\begin{proof}[Proofs of Theorems \ref{ftcrit} and \ref{expcrit}]
	Notice that if $\lambda>\lambda_c$ then delSIR has a large epidemic wpp. Since the epidemic size in evoSIR couples to be larger, we also have a large epidemic in evoSIR. If $\lambda< \lambda_c$ then the summable bound in Lemma \ref{rtail} implies that $R \leq C_1 \log n$ for all large enough $n$ whp. Thus, the probability of a large epidemic converges to 0. To generalize to exponential infection times we repeat a similar argument with 
$$
\tau_e = \frac{\lambda}{\lambda+1} \qquad \alpha =  1 - \frac{\tau^e_r}{\tau_e} = \frac{\rho}{\lambda+1+\rho}.
$$
\end{proof}


\section{Proof of Theorem \ref{ftperc} and \ref{exprwcrit} } \label{ssec:pfth3}

Consider delSIR and evoSIR on $G= G(n,\mu/n)$ with deletion (or rewiring) at rate $\rho$, infection rate $\lambda$, and infections last for fixed time 1. Let 
$$
\tau^f_r = \frac{\lambda}{\lambda+\rho} ( 1 - \exp^{-(\lambda+\rho)} )
$$
be the probability that an infection is transmitted from an $I$ at $x$ to a $S$ at $y$ before $x$ becomes healthy or $y$ rewires the connection. By Theorem \ref{ftcrit}, we must have  $\tau \mu >1$ for a large infection to be possible. Start with 1 infected and all other vertices susceptible. Let $B_d = B_d(n)$ be the event that the ending number of infected sites is at least $(1-z_0)n/2$  in delSIR. Here $z_0$ is the fixed point of the generating function giving the epidemic size for delSIR at \eqref{hatG}. Similarly, define $B_e$ for evoSIR.
 In order to prove Theorem \ref{ftperc} we need to show

$$
P(B_e) = P(B_d)+o(1).
$$	

\begin{proof} Clearly $P(B_d) \leq P(B_e)$. To compute the size of the epidemic starting from a single infected, we let the 
the active or infected set $A_0=\{1\}$, the unexplored or susceptible set $U_0 =\{2, \ldots n \}$, and the removed set $R_0=\emptyset$. Let $\eta_{i,j} \eta_{j,i}$ be independent and $=1$ with probability $\mu/n$, $=0$ otherwise. $\eta_{i,j}=1$ if (and only if) there is an edge from $i$ to $j$. For $i \neq j$ let $\zeta_{i,j}=\zeta_{j,i}$ be independent and $=1$ with probability $\tau^f_r$, $=0$ otherwise. If $\zeta_{i,j}=1$ an infection at $i$ is transmitted to $j$. At time $t$ if $A_t \neq\emptyset$ we pick an $i_t \in A_t$ and update the sets as follows. 
\begin{align*}
{\cal R}_{t+1} & =  {\cal R}_t \cup \{ i_t \} \cr
{\cal A}_{t+1} & =  {\cal A}_t - \{ i_t \} \cup \{ y \in {\cal U}_t : \eta_{i_t,y} = \zeta_{i_t,y} = 1 \} \\
{\cal U}_{t+1} & =  {\cal U}_t - \{ y \in {\cal U}_t : \eta_{i_t,y} = \zeta_{i_t,y} = 1 \})
\end{align*}
When ${\cal A}_t=\emptyset$ we have found the cluster containing 1.

Let $A_t= |{\cal A}_t|$ be the number of active sites at time $t$. Consider the time $r= \beta \log n$ defined in the proof of \cite[Theorem 2.3.2]{RGD}. Step 2 of that proof tells us that there are constants $\gamma , C >0$ such that
$$
P(0 < A_r < \gamma \log n) = o(n^{-1}).
$$
Let $F_0 = \{A_r = 0\}, F_1 = \{0 < A_r < \gamma \log n\}$ and $F_2 = \{A_r \geq \gamma \log n\}$. Decomposing $B_d$ into three parts
$$
P(B_d) = \sum_{i=0}^2 P( B_d \mid F_i) P(F_i) = P(B_d \mid F_2)P(F_2) + o(1).
$$
Step 4 of \cite[Theorem 2.3.2]{RGD} implies that $P(B_d \mid F_2) = 1- o(1)$ so
$$
P(B_d) = P(A_r \geq \gamma \log n) + o(1).
$$

Suppose we perform the exploration process for evoSIR using the same $\eta_{i,j}$ and $\zeta_{i,j}$. To account for rewiring, we introduce a sequence of independent random variables $\chi_{i,j} = \chi_{j,i}$ that are $=1$ with probability 
$$
 \frac{\rho}{\lambda+\rho} ( 1 - \exp^{-(\lambda+\rho)} ),
$$
and 0 otherwise. Note that the $\chi_{j,i}$ are never 1 when $\zeta_{i_t,y} = 1$. When $ \eta_{i_t,y} = \chi_{i_t,y} = 1$, $y$ rewires to a vertex chosen at random. 
Let $A'_r$ be the set of active sites at time $r$. Let $D$ be the maximum degree of vertices in $G(n,\mu/n)$.
It follows from Lemma \ref{ldev} that if we pick $C$ large enough then 
$$
P(D > C \log n ) \leq n^{-2}.
$$
Let $S_r$ be the total number of edges have been rewired up to time $r$ in the exploration process. It is trivial that 
$$
S_r \leq D r.
$$
Using the last two results with the fact that $r=\beta \log n$, we have 
$$
P(S_r \geq C \beta \log^2 n) \leq n^{-2}.
$$
Let $R'_t$ is the set of removed sites at time $t$. This grows by 1 at each step the exploration process survives so $A'_t+R'_t$ grows by at most $D$ at each step. 
and it follows that
$$
P(A'_r+R'_r \geq C \beta \log^2 n) \leq n^{-2}.
$$
From this it follows that the probability that no edge counted by $S_r$ is rewired back to a vertex counted by $A'_r + R'_r$ is 
$$
\ge \left( 1- \frac{C \beta \log^2 n }{n} \right)^{ C \beta \log ^2 n} = 1- O(\log^4 n /n).
$$
From this it follows that $P(A'_r = A_r) = 1 - O(\log^2n /n)$ and the desired result follows.
\end{proof}

\mn
{\bf Remark.} To extend the proof to continuous time we set
$$
P( \eta_{i,j}=1) = \frac{\lambda}{\lambda+1+\rho} \text{ and } P( \chi_{i,j}=1) = \frac{\rho}{\lambda+1+\rho}.
$$


\section{A procedure for studying epidemic size}
In this section we expand upon the framework used by Martin-L\"of to study SIR. This leads to a proof of Theorem \ref{sizeexp}, and also more details regarding our attempt to adapt an ODE approach to evoSIR from Section \ref{sec:attempt}. 

\subsection{Fixed time epidemic} \label{sec:MLfixed}

Our first step is to delete edges that the infection will not cross. This results in an Erd\H os-Renyi graph $G(n,\bar\mu/n)$ with $\bar\mu = \mu\tau$.  Recall the construction with $R_t$, $A_t$ and $U_t$ from Section \ref{sec:ODE}, but now when $A_t=\emptyset$, pick an $i_t$ from $U_t$ to continue the construction.
Let $\Delta U_t = U_{t+1} - U_t$. The individual $i_t$ chosen at time $t$ is connected to each individual in $U_t$ with probability $\bar\mu/n$.
The number of new connections is binomial($U_t,\bar\mu/n)$ so
\begin{align*}
E( \Delta U_t | {\cal F}_t ) & =  - U_t \frac{\bar\mu}{n} \cr
\noalign{\smallskip}
\var( \Delta U_t | {\cal F}_t ) & =  U_t \frac{\bar\mu}{n} \left( 1- \frac{\bar\mu}{n} \right)
\end{align*}
Let $X_t = (1-\bar\mu/n)^{-t} U_t$ and observe that
$$
E(X_{t+1}\mid {\cal F}_t) = (1-\bar\mu/n)^{-(t+1)} E(U_{t+1}\mid {\cal F}_t) = (1-\bar\mu/n)^{-(t+1)} (1-\bar\mu/n) U_t = X_t.
$$
So, $X_t$ is a martingale. Multiplying by  a constant or adding a number does not change the martingale property so $Y_t = X_t/n - 1$ is also a martingale. Using the conditional expectation version of $E(X-EX)^2 = EX^2 - (EX)^2$ (Theorem 5.4.7 in \cite{PTE4})
\begin{align*}
E( Y^2_{t+1}-Y^2_t \mid {\cal F}_t) & = E( (Y_{t+1} - Y_t)^2  \mid {\cal F}_t ) \\
& = \frac{1}{n^2} E((X_{t+1} - X_t)^2  \mid {\cal F}_t) \\
& = \frac{1}{n^2} (1-\bar\mu/n)^{-2(t+1)} E((U_{t+1}-(1-\bar\mu/n)U_t)^2 \mid {\cal F}_t) \\
& = \frac{1}{n^2} (1-\bar\mu/n)^{-2(t+1)}  \frac{U_t}{n} \bar\mu ( 1- \bar\mu/n)
\end{align*}

Using $U_t \in [0,n]$
$$
E( Y^2_{t+1}-Y^2_t \mid {\cal F}_t) \le C/n^2
$$
Taking expected value $E( Y^2_{t+1}-Y^2_t) \le C/n^2$ and it follows that $E(Y^2_n) \le C/n$. Using the $L^2$ maximal inequality now
$$
E \left( \max_{0\le t \le n} Y_t^2 \right) \le 4C/n
$$
Using Chebyshev's inequality now and filling in the definitions of $Y_t$ and $X_t$
$$
P \left( \max_{0\le m \le n}| (1-\bar\mu/n)^{-t}U_t/n - 1| > n^{-1/2 +\ep} \right) \le 4Cn^{-2\ep}
$$
and we conclude that
$$
U_{[ns]}/n \to e^{-\bar\mu s} \quad \hbox{uniformly on $[0,1]$}.
$$

$R_t=t$ and $U_t = n - A_t - R_t$, so when $U_t=n-t$ we must have $A_t=0$. This may happen several times at the beginning, but eventually
we will select a member of the giant component. Using this with previous observation we find that the size of the giant component is
the solution of 
$$
e^{-\mu s} = 1 -s
$$
Since the size of the giant component is $1-\rho$ where $\rho$ is the extinction probability, we have a very lengthy derivation of
the fact that $\rho$ is the fixed point of the generating function: $\exp(-\mu(1-\rho))=\rho$. The advantage of the new approach is that
by using formulas for the infinitesimal mean and variance one can show that 
$$
\sqrt{n} \left( \frac{U_{[ns]}}{n} - e^{-\bar\mu s} \right)
$$
has a Gaussian limit, and derive a central limit theorem for the size of the epidemic. For more details see \cite[Section 2.5]{RGD} or \cite{ML}.

\begin{figure}[tbp] 
  \centering
  \includegraphics[bb=53 58 737 555,width=4in,height=2.91in,keepaspectratio]{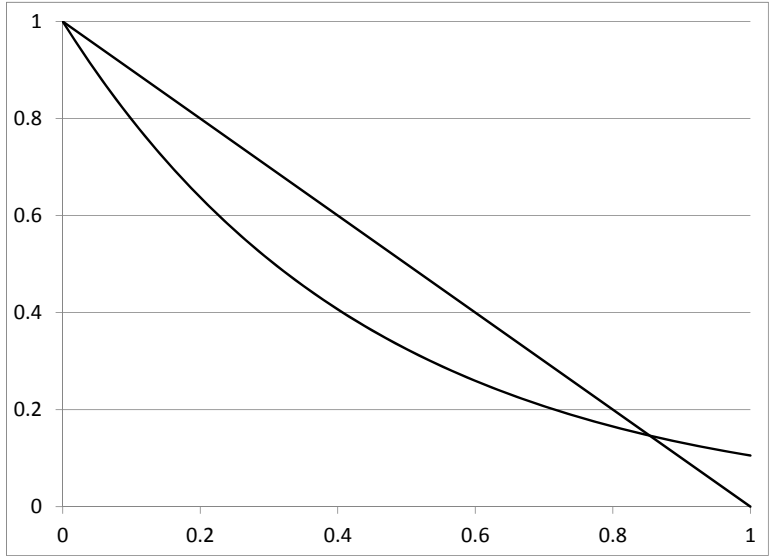}
  \caption{Martin-L\"of argument in discrete time.}
  \label{fig:ML}
\end{figure}

\subsection{Exponential time and proof of Theorem \ref{sizeexp}} \label{sec:MLexp}

When we reveal the neighbors of the $s$th vertex, we assign it an exponentially distributed infection time $T_s$ and 
define the transmissibility by $\tau_s= 1 -e^{\lambda T_s}$. So, if the value of $\tau_s$ is added to the $\sigma$-field ${\cal F}_s$ 
and the result is called ${\cal F}^+_s$, then we have
\begin{align*}
E( \Delta U_s | {\cal F}^+_s ) & =  - U_s \frac{\mu\tau_s}{n} \cr
\noalign{\smallskip}
\var( \Delta U_s | {\cal F}^+_s ) & =  U_s \frac{\mu\tau_s}{n} \left( 1- \frac{\mu\tau_s}{n} \right),
\end{align*}
since given $\tau_s$ the infections are independent. Note that now $\mu$ is the mean degree in the original Erd\H os-Renyi graph.
Rearranging the first equation 
$$
E(U_{s+1}|{\cal F}^+_s) = U_s (1 - \mu\tau_s/n).
$$
Let $\Pi_s = \prod_{r=1}^{s-1} (1- \frac{\mu\tau_r}{n})^{-1}$ and $X_s = U_s\Pi_s$.
To check that $X_s$ is a martingale, note that $\Pi_{s+1}$ is ${\cal F}^+_s$ measurable so
$$
E(X_{s+1}|{\cal F}^+_s) = \Pi_{s+1} E(U_{s+1}|{\cal F}^+_s)  = \Pi_{s+1} U_s \left(1 - \frac{\mu\tau_s}{n}\right) = \Pi_s U_s = X_s
$$
To compute the variance now we note that
\begin{align*}
E((X_{s+1}-X_s)^2|{\cal F}^+_s) & = E((X_{s+1}-X_s)^2|{\cal F}^+_s) \\
& = \Pi_{s+1}^2 E([U_{s+1} - U_s(1-\mu\tau_s/n)]^2| |{\cal F}^+_s) \\
&= \Pi^2_{s+1} \var(\Delta U_s| {\cal F}^+_s) = \Pi^2_{s+1}\frac{ U_s}{n} \mu\tau_s (1-\mu\tau_s/n)
\end{align*}
Since $0 \le \tau_s \le 1$, $\Pi_s \le \exp(\mu s/n)$.
If we let $Y_s = X_s/n - 1$ then
$$
E(Y_{s+1}^2 - Y_s^2|{\cal F}^+_s ) = E((Y_{s+1} - Y_s)^2|{\cal F}^+_s ) \le C/n^2
$$
Using the $L^2$ maximal inequality and Chebyshev's inequality
$$
P \left( \max_{0\le t \le n}| \Pi_t U_t/n - 1| > n^{-1/2 +\ep} \right) \le 4Cn^{-2\ep}
$$
To extract the limit from this we note that
$$
\log \Pi_{ns}  = - \sum_{r=1}^{ns-1} \log (1- \mu\tau_r/n) 
= (1/n) \sum_{r=1}^{ns-1} \mu\tau_r + o(1) \to s \mu E\tau_r
$$
where in the second step we have used $\log(1-x) = - x + O(x^2)$ and in the third we have used the law of large numbers.
From this it follows that 
$$
U_{[ns]}/n \to e^{-s\mu E\tau } \quad \hbox{uniformly on $[0,1]$}.
$$
which proves Theorem \ref{sizeexp}.

\subsection{Fixed time with rewiring} \label{sec:rewire}

To take into account the rewiring in the fixed time case, we will let $v_s$ be the average degree of unexplored vertices at time $s$.
Repeating the reasoning in Section \ref{sec:MLfixed} to differential equations 
\begin{align*}
\frac{du_s}{ds}& = - v_s \tau u_s (1-\alpha) \\
\frac{dv_s}{ds}& = v_s\tau u_s \alpha 
\end{align*}
where $\alpha$ is the probability a rewiring prevents an infection. The same proof works for exponential infection times with $\tau$ replaced by $E\tau$. 

Combining the two equations
$$
\frac{d}{ds}[ \alpha u_s + (1-\alpha) v_s ] = 0
$$
i.e., $\alpha u_s + (1-\alpha) v_s$ is constant and hence equal to its value at time 0, $\alpha + (1-\alpha) \mu$. 
Rearranging 
$$
\alpha u_s + (1-\alpha) v_s = \alpha + (1-\alpha) \mu
$$
we have $v_s = \mu + \alpha(1-u_s)/(1-\alpha)$.
Using this in the first equation
$$
\frac{du_s}{ds} = - u_s\tau \left[\mu + \frac{\alpha(1-u_s)}{1-\alpha} \right].
$$
To solve this write it as
\beq
\frac{du}{ds} = - u ( A- Bu)
\label{logistic}
\eeq
where $A = \tau[\mu(1-\alpha) + \alpha]$ and $B = \tau\alpha$.
Cross-multiplying
$$
-ds = \frac{du}{u(A-Bu)} = \frac{1}{A} \frac{du}{u} + \frac{B}{A}\frac{du}{A-Bu}.
$$
Integrating gives
\begin{align*}
-t + c_0 &= \frac{\log u}{A} - \frac{\log(A-Bu)}{A} \\
-At + c_1 &= \log\left( \frac{u}{A-Bu} \right) \\
 \frac{u}{A-Bu} &= c_2 e^{-At},
\end{align*}
so we have $u e^{At} = c_2(A-Bu)$. Solving gives
$$
u = \frac{Ac_2}{e^{At} + Bc_2}.
$$
We cannot choose $c_2$ to satisfy the initial condition, so we set $c_2=1$ and pick $t_0$ so that 
$e^{At_0} = (B-A)$. This gives
\beq
u = \frac{A}{B + (A-B) e^{At}}.
\label{soln}
\eeq

To check \eqref{soln} we differentiate
\begin{align*}
u'(t) & = -\frac{A(A-B)(Ae^{At})}{(B + (A-B) e^{At})^2}  = - u \frac{(A-B)(Ae^{At})}{B + (A-B) e^{At}} \\
A-Bu & = \frac{A (B + (A-B)) e^{At} - BA}{(B + (A-B) e^{At}}.
\end{align*}



\begin{thebibliography}{99}

\bibitem{Allen}
Allen, Linda J. (2003)
{\it An introduction to stochastic processes with applications to biology.}
Prentice Hall

\bibitem{BM} 
Barbour, A.D., and Mollison, D. (1990)
Epidemics and random graphs.
{\it Stochastic Processes in Epidemic Theory.}
Springer lecture Notes in Biomatemtics.

\bibitem{BaSly}
Basu, R., and Sly, A. (2017)
Evolving voter model on dense random graphs. 
{\it Ann. Appl. Probab.} 27, 1235--1238 

\bibitem{BJS}
Britton, T., Juher, D., and Saldana, J. (2016)
A network model with preventive rewiring: Comparative of the initial phase.
{\it Bulletin Math. Biology.} 78, 2427--2454 

\bibitem{BBLS}
Ball, F., Britton, T., Leung K.Y., and Sirl, D. (2018)
A stochastic SIR epidemic model with preventive dropping of edes.
arXiv: 1808.08100 

\bibitem{Decr}
Decreusefond, L., Dhersin, J.-S., Moyal, P., and Tran, V.C. (2012)
Large graph limit for an SIR process in a random network with heterogeneous connectivity.
{\it Ann. Appl. Probab.} 22, 541--575

\bibitem{DM}
Draief, M., and Massoulie, L. (2010)
{\it Epidemics and Rumors in Complex Networks.}
London Mathematical Society

\bibitem{RGD}
Durrett, R. (2007) 
{\it Random Graph Dynamics.} Cambridge U. Press

\bibitem{PTE4}
Durrett, R. (2010)
{\it Probability: Theory and Examples.} Fourth Edition (2010), Cambridge U. Press

\bibitem{evov}
Durrett, R., Gleeson, J., Lloyd, A., Mucha, P., Shi,F., Sivakoff, D., Socloar, J., and Varghese, C. (2012)
Graph fission in an evolving voter model. 
{\it Proc. Natl. Acad. Sci.} 109, 3682--3687

\bibitem{Grass}
Grassberger, P. (1983)
on the critical behavior of the general epidemic process and dynamical percolation.
{\it Mathematical Biosciences.} 63, 157--172

\bibitem{GDB}
Gross T, Dommar D'Lima CJ, Blasius B. (2006) Epidemic dynamics on an adaptive network.
{\it Phys. Rev. Letters.} 96, paper 208701

\bibitem{GB}
Gross T, Blasius B (2008) Adaptive coevolutionary networks: a review.
{\it J. Royal Soc. Interface} 5, 259--271

\bibitem{HCTG}
Herrera JL, Cosenza MG, Tucci K, Gonz\'alez-Avella JC (2011)
General coevolution of topology and dynamics in networks. arXiv:1102.3647

\bibitem{HN}
Holme P, Newman MEJ (2006) Nonequilibrium phase transition in the coevolution of
networks and opinions. {\it Phys. Rev. E} 74:056108

\bibitem{Janson}
Janson, S., Luczak, M., and Windbridge, P. (2014)
Law of large numbers for the SIR epidemic on a random graph with given degrees.
{\it Rand. Struct. Alg.} 45, 726--763

\bibitem{JRS}
Juher, D., Ripolli, J., and Saldana, J. (2010)
Outbreak analysis of an SIS epidemic model.
{\it J. Math. Biol.} 67, 411--432

\bibitem{Kuul}
Kuulasmaa, K. (1982)
The spatial general epidemic and locally dependent random graphs.
{\it J. Applied Probab.} 19, 745--758

\bibitem{LBSB}
Leung, K.Y., Ball, F., Sirl, David, and Britton, T. (2018)
Individual preventive measures during an epidemic may have negative population-level consequences
arXiv:1805.05211

\bibitem{ML}
Martin-L\"of, A. (1986) Symmetric sampling procedures, general epidemic processes and their threshold limit theorems.
{\it J. Appl. Prob.} {\bf 23}, 265--282

\bibitem{Miller}
Miller, J.C. (2011)
A note on a paper by Erik Volz: SIR dynamics in random networks.
{\it J. Math. Bio.} 62, 349--358

\bibitem{MR}
Molloy, M., and Reed, B. (1995) 
A critical point for random graphs with a given degree sequence. 
{\it Random Struct. Alg.} {\bf 6}, 161--179

\bibitem{Newman2002}
Newman, M.E.J. (2002) Spread of epidemic disease on networks. 
{\it Phys. Rev. E.} 66, paper 016128

\bibitem{NSW}
Newman, M.E.J., Strogatz, S.H., and Watts, D.J. (2001) 
Random graphs with arbitrary degree distributions and their applications. 
{\it Phys. Rev. E.} {\bf 64}, paper 026118


\bibitem{SBM}
Schwarzkopf, Y., R\'akos, A., and Mukamel, D. (2010)
Epidemic spreading in evolving networks 
arXiv:1001.3878

\bibitem{SSP}
van Segroeck S, Santos FC, Pacheco JM (2010) Adaptive contact networks change effective
disease infectiousness and dynamics. {\it PLoS Comp. Biol.} 6(8): paper e1000895

\bibitem{VEM}
Vazquez F, Egu\'iluz VM, San Miguel MS (2008) Generic absorbing transition in coevolution dynamics.
{\it Phys. Rev. Letters.} 100, 108702

\bibitem{Volz}
Volz, E. (2008)
SIR dynamics in random graphs with heterogeneous connectivity.
{\it J. Math. Biol.} 56, 293--310

\bibitem{VM1}
Volz E, Meyers LA (2007) Susceptible-infected recovered epidemics in dynamic contact networks.
{\it Proc. Roy. Soc. B} 274, 2925--2933

\bibitem{VM2}
Volz E, Meyers LA (2009) Epidemic thresholds in dynamic contact networks.
{\it J. Roy. Soc. Interface.} 6, 233--241

\bibitem{Zanette}
Zanette, D.H. (2007) Coevolution of agents and networks in an epidemiological model.
arxiv:0707.1249

\bibitem{ZanGus} 
Zanette, D.H., and Gusm\'an, S.R. (2007)
Infection spreading in a population with evolving contacts.
arXiv:0711.0874

\end{thebibliography}
\end{document}